\documentclass{article}

\usepackage{arxiv}
\usepackage{natbib}
\usepackage{url}

\usepackage{graphicx}
\usepackage{amsmath,amssymb,amsthm}
\usepackage[usenames,dvipsnames]{xcolor}
\usepackage{xfrac}
\usepackage{booktabs,multirow,paralist,adjustbox,array} 
\usepackage{pifont}
\usepackage{afterpage}
\newtheorem{assumption}{Assumption}
\newtheorem{definition}{Definition}
\newtheorem{lemma}{Lemma}
\newtheorem{theorem}{Theorem}
\usepackage{bbm}
\usepackage{algorithm}
\usepackage{algorithmic}

\newcommand{\Stepsize}[1]{\mathcal{A}_{#1}} 
\newcommand{\stepsize}[2]{\alpha_{#1}^{#2}}
\newcommand{\stepsizebar}{\bar{\alpha}}
\newcommand{\constG}{\kappa}

\newcommand{\iterTrue}{I_k^{\text{true}}}
\newcommand{\iterSucc}{I_k^{\text{succ}}}
\newcommand{\iterLarge}{I_k^{\text{large}}}
\newcommand{\iterLargebar}{I_k^{\text{large+}}}
\newcommand{\pastone}{\mathcal{F}_{k-1}}
\newcommand{\pastonei}{\mathcal{F}_{i-1}}

\def\EE{\mathbb{E}}
\newcommand{\R}{{\mathbb R}} 

\newcommand{\Prob} {\mathbb{P}} 

\newcommand{\Alpha}{{\cal A}}
\newcommand{\norm}[1]{\left\Vert#1\right\Vert}

\newcommand\blfootnote[1]{%
  \begingroup
  \renewcommand\thefootnote{}\footnote{#1}%
  \addtocounter{footnote}{-1}%
  \endgroup
}

\title{Stochastic ISTA/FISTA Adaptive Step Search Algorithms for Convex Composite Optimization}

\author{Lam M. Nguyen$^{1}$, Katya Scheinberg$^{2}$, Trang H. Tran$^{3*}$ \\
$^{1}$ IBM Research, Thomas J. Watson Research Center, Yorktown Heights, NY, USA \\
$^{2}$ H. Milton Stewart School of Industrial and Systems Engineering,  
Georgia Institute of Technology, 
Atlanta, GA, USA \\
$^{3}$ School of Operations Research and Information Engineering, Cornell University, Ithaca, NY, USA\\
\\
\texttt{LamNguyen.MLTD@ibm.com}, \texttt{katya.scheinberg@isye.gatech.edu}, \texttt{htt27@cornell.edu}
}

\begin{document}

\maketitle

\begin{abstract}
We develop and analyze stochastic variants of ISTA and a full backtracking FISTA algorithms \cite{Beck-Teboulle,Scheinberg-Goldfarb-Bai}
for composite optimization without the assumption that stochastic gradient is an unbiased estimator. This work extends analysis of inexact fixed step ISTA/FISTA in \cite{Schmidt-Roux-Bach}  to the case of stochastic gradient estimates and adaptive step-size parameter chosen by backtracking. It also extends the framework for analyzing stochastic line-search method in \cite{Cartis-Scheinberg} to the proximal gradient framework as well as to the accelerated first order methods. 
\end{abstract}

\blfootnote{$^{*}$ Corresponding author \\ 
This work was partially supported by NSF Grants CCF 20-08434, CCF 21-40057 and ONR award N00014-22-1-215.}

\section{Introduction}
In this paper we  consider, what is often referred to as,  a convex composite optimization problem of the form: 
\begin{align*}
\min_{x \in \R^n} F(x) = f(x) + h(x),  
\end{align*}
under the following assumptions: 
\begin{assumption}\label{assumpt:convex_smooth}
 Function  $f$ is convex and $L$-smooth, function $h$ is proper, closed and convex.
\end{assumption}
\begin{assumption}\label{assumpt:bounded_below}
There exists a global minimizer of $F(x)$, $x^*$,  such that $F(x^*)=F^* > -\infty$. 
  \end{assumption}
  
This class of problems has been extensively studied in recent literature due to broad applications in image processing and data science, e.g., see \cite{Chambolle-DeVore-Nam-YongLee-Lucier,Combettes-Pesquet,Hansen-Nagy-OLeary,Nesterov-Nemirovski,Wainwright}. When $h(x)$ has particularly simple form, such that a proximal operator
\begin{align}
\text{prox}_h (y) = {\arg \min}_{x \in \R^n} \left\{h(x) +\frac{1}{2}\|x-y\|^2\right\},  \label{eq:prox_oper}
\end{align}
 can be computed efficiently, a popular and effective class of algorithms for these problems is proximal gradient method, also known as forward-backward iterative algorithms \cite{Bruck,Combettes-Pesquet,Combettes-Wajs,Passty} within the general framework of splitting methods \cite{Facchinei-Pang}. 
Among them, iterative shrinkage-thresholding algorithms (ISTA) \cite{Chambolle-DeVore-Nam-YongLee-Lucier,Combettes-Wajs,Daubechies-Defrise-DeMol,Figueiredo-Nowak} and its accelerated version FISTA \cite{Beck-Teboulle} are probably the most popular.

One of the features of the algorithms in \cite{Beck-Teboulle} is the adaptive step-size parameter selection using a sufficient reduction condition. 
 In practical deterministic optimization, it is very common to try out step-size  parameter values and select them based on the achieved progress rather than a guess of a global constant $L$, which can be too conservative. Three of such adaptive algorithmic frameworks are line-search, trust-region and adaptive cubic regularization methods  \cite{Cartis-Gould-Toint,Nocedal-Wright}. Recently, all these methods have been analyzed in stochastic setting and shown to maintain favorable convergence properties under various conditions on stochastic function oracles \cite{Blanchet-Cartis-Menickelly-Scheinberg,Cartis-Scheinberg,Jin-et-al,Paquette-Scheinberg}. In particular, it is shown that when stochastic gradient and function estimates satisfy certain accuracy requirements with sufficiently high (but fixed) probability, then the iteration complexity of each method matches that of their deterministic counterpart. 

On the other hand, there have not been similar stochastic versions  {developed for} accelerated gradient methods, such as FISTA.
One obstacle  is that they do not easily lend themselves to fully adaptive modes.  In particular, the  step-size  parameter of FISTA algorithm in \cite{Beck-Teboulle} can be decreased on any given iteration, but it cannot be increased between iterations, as is done by the line-search, trust-region and adaptive cubic regularization methods. While undesirable in practice, this does not cause difficulties in the analysis of deterministic accelerated methods, since the step-size  parameter remains bounded from below by a positive constant throughout the application of the algorithm. This is the consequence of Lipschitz smoothness of $f(x)$. 
To guarantee such a lower bound in the stochastic case, a line-search algorithm in \cite{Vaswani-et-al} and a similar proximal gradient algorithm in  \cite{Franchini-Porta-Ruggiero-Trombini} have to utilize {a very restrictive type of sufficient decrease condition and impose strong assumptions on the stochastic approximation}. It is shown in \cite{Jin-Scheinberg-Xie} that such decrease conditions results in an inferior theoretical results as well as practical performance. We discuss more details in Section \ref{subsec:outline_algo}. If these restrictions are removed, then step-size  parameter can decrease to an arbitrary small value and only by allowing it to increase can complexity bounds be established.  
 
An additional difficulty is that  the complexity analysis of accelerated first order methods is quite different from those of line-search, trust-region and adaptive cubic regularization methods. For the latter, complexity analysis relies on the fact that the objective function decreases by at least some fixed amount on each successful step (or in expectation). Analysis then involves martingale theory to derive bounds on the expected number of iterations that the algorithm performs until certain accuracy is reached. 
This decrease, however, does not necessarily hold for accelerated methods. Thus until this paper it was unclear if ideas from  \cite{Blanchet-Cartis-Menickelly-Scheinberg,Cartis-Scheinberg,Jin-Scheinberg-Xie,Paquette-Scheinberg} can be extended to methods such as FISTA. 

Here, we use the full backtracking version of FISTA, which  was proposed and analyzed in 
  \cite{Scheinberg-Goldfarb-Bai}. We  extend framework from \cite{Cartis-Scheinberg}  to ISTA in \cite{Beck-Teboulle} and 
the full backtracking variant of FISTA in \cite{Scheinberg-Goldfarb-Bai} and show that under certain conditions on the stochastic gradient error the algorithms retain their deterministic iteration complexity bounds, in expectation, {and up to constants}. In this paper, as in  \cite{Cartis-Scheinberg}, we assume that only the gradient $\nabla f(x)$ is approximated by a stochastic estimator, the function values $f(x)$ and $h(x)$, as well as the proximal step are assumed to be exact. 
  { This type of setting can be motivated by model based algorithms in derivative free optimization, where first (and second) derivative approximations are computed approximately by  some finite difference or interpolation schemes based on random set of sample points. Such methods have been popular in the recent literature \cite{Blanchet-Cartis-Menickelly-Scheinberg,Chen-Menickelly-Scheinberg}. But, more generally, while having access to exact function values (exact zeroth-order orcale) is a strong assumption, it can be viewed as the first step in developing expected complexity analysis for any methods with adaptive step size selection and random first-order oracles.  
For example, the analysis in \cite{Cartis-Scheinberg} is based on exact zeroth-order oracle, but later was extended in \cite{Berahas-Cao-Scheinberg} to the case of inexact function estimates with bounded errors, and in \cite{Jin-Scheinberg-Xie} to stochastic function estimate with subexponential error. Similarly analysis of a trust region method in \cite{Bandeiraetal2014,gratton2018complexity} was derived for the case of exact zeroth-order oracle and later extended to stochastic oracles in \cite{Blanchet-Cartis-Menickelly-Scheinberg,Chen-Menickelly-Scheinberg}. In the earlier version of this paper \cite{Tran-Nguyen-Scheinberg} we presented the analysis of the smooth case (with $h(x)\equiv 0$) under the assumption that the function evaluation error is not zero but decays suitable fast. We chose to simplify the analysis in this paper, by focusing on the error in the gradient only. Extension of this analysis with suitably decaying function evaluation error is relatively simple, but as part of future research we would like to consider a variety of assumptions on the zeroth order oracle. }

  {The results of this paper can also be viewed as an extension of analysis in  \cite{Schmidt-Roux-Bach} developed for inexact variants of ISTA and FISTA under the assumption that the gradient of $f$ and the proximal operator \eqref{eq:prox_oper} are computed with some (deterministic) error. Here we  assume that  the proximal operator is exact, and we only relax the assumption that the error in the gradient is deterministic. In  \cite{Schmidt-Roux-Bach}  fixed  ISTA and FISTA are analyzed for a fixed step size, thus sufficient decrease condition is not used and function values are not computed or estimated. Further extension of the analysis presented here to the case of inexact stochastic proximal operators in also of interest and is subject for future research. }

There is significant literature on stochastic proximal gradient methods, e.g. \cite{Franchini-Porta-Ruggiero-Trombini,Hu-Pan-Kwok,Kulunchakov-Mairal,Lan,Nitanda,Poon-Liang-Schoenlieb,Xiao-Zhang}, proximal methods with  variance reduction \cite{Defazio-Bach-Lacoste-Julien,Pham-Nguyen-Phan-Tran-Dinh,Xiao-Zhang} and accelerated versions \cite{Hu-Pan-Kwok,Lan,Nitanda}. 
In all these works stochastic  estimates of  $\nabla f(x)$ are assumed to be unbiased. 

Stochastic accelerated gradient   methods  for smooth optimization (when $h(x)\equiv 0$) have also been analyzed, e.g. \cite{Cohen-Diakonikolas-Orecchia,Schmidt-Roux,Vaswani-Bach-Schmidt}.   These methods also make use of the assumption that an unbiased estimator of $\nabla f(x)$ is computed at each step and the sequence of  {step-sizes} is fixed or predetermined. Since these algorithm do not utilize a sufficient decrease condition, the objective function values $f(x)$ is not estimated.

In summary, we assume that a stochastic (but not necessarily unbiased) estimate of  $\nabla f(x)$ can be computed, but the values of $f(x)$ and $h(x)$ can be computed exactly. These assumptions are quite  {different than} what is typically studied in the literature.  Our results can be viewed as extensions of the following papers:
\begin{itemize}
\item  We extend analysis of inexact ISTA/FISTA in \cite{Schmidt-Roux-Bach} to the case when the error in the gradient computation is stochastic and thus is not deterministically bounded. 
\item We extend the full backtracking  variant of FISTA in  \cite{Scheinberg-Goldfarb-Bai} to the case of stochastic gradient estimation. 
\item We extend the framework for  search methods with generic stochastic oracles in  \cite{Berahas-Cao-Scheinberg,Cartis-Scheinberg} to accelerated and proximal methods. 

\end{itemize}
  
The remainder of the paper is organized as follows. Section \ref{sec:review} contains a discussion about how our results compare to the most relevant existing results. In Section \ref{sec:ista} we introduce stochastic ISTA algorithm, discuss the properties of the underlying stochastic process and derive the expected complexity result for ISTA. In Section \ref{sec:fista} we build on the results from Section \ref{sec:ista} and derive the complexity analysis for the stochastic backtracking FISTA algorithm. Final remarks are presented in Section 
\ref{sec:summary}.

\section{A Brief Review of Relevant Results}\label{sec:review}

There exists substantial research analyzing gradient methods in the context of deterministic  and stochastic noise in gradient computations. It is well known that the convergence rate of a stochastic gradient method is usually worse than that of it deterministic counterpart \cite{Devolder-Glineur-Nesterov,Hu-Pan-Kwok,Lan}, unless  some conditions on vanishing variance of the noise is imposed \cite{Cohen-Diakonikolas-Orecchia,Schmidt-Roux,Schmidt-Roux-Bach}. 
Practical evidence demonstrates that inexact accelerated gradient algorithms  are sensitive to noise levels \cite{Devolder-Glineur-Nesterov,Liu-Belkin}. Theoretical analysis confirms this phenomenon and further suggest that the inexact accelerated algorithms need a higher order of accuracy in gradient estimators than their non-accelerated versions to maintain their fast convergence rates \cite{Devolder-Glineur-Nesterov,Schmidt-Roux-Bach}.

In this paper, we extend the inexact ISTA/FISTA analysis in \cite{Schmidt-Roux-Bach} to the stochastic errors in gradient computation, with a full backtracking framework. We do not require the gradient estimator to be unbiased. While our vanishing noise conditions have the same order as the deterministic errors in \cite{Schmidt-Roux-Bach}, our assumptions are in expectation. Other than this assumption, since our algorithm uses backtracking, we need an additional Assumption \ref{assumpt:key} to that sufficient decrease can be achieved with probability at least $p>1/2$. Although not directly comparable, this assumption is weaker from practical perspective than 
the  vanish noise condition, hence our assumptions are not stronger than in  \cite{Schmidt-Roux-Bach}.

While  {there are} multiple results on accelerated  proximal  stochastic gradient methods, (e.g. \cite{Hu-Pan-Kwok,Lan}) our assumptions and convergence results are different from all prior  methods that we are aware of.  In particular, our analysis bounds the expected number of iterations to achieve $F(x_k)- F^*\leq \epsilon$ as a function of $\epsilon$. We specifically show that this bound is of the same order as the deterministic complexity. All other results  bound
 $\EE[F(x_k)- F^*]$ as a function of iteration number $k$.  {In the appendix, we demonstrate that bounds for expected complexity and expected convergence rate of general stochastic processes are not comparable.}
 
 In addition, as mentioned,  we do not assume that estimates of $\nabla f(x)$ are unbiased, although the bias is assumed to decay with $k$. Most other works on stochastic proximal gradients assume unbiased gradient estimates,  except for a small class that use incremental gradient and shuffling data scheme \cite{Bertsekas,Mishchenko-Khaled-Richtarik}.
 

Specific line of work that closely related to our paper is analysis of gradient methods and accelerated methods under the strong growth condition (SGC) \cite{Schmidt-Roux,Vaswani-Bach-Schmidt}. 
It is easy to see that the strong growth condition implies our Assumption \ref{assumpt:key} when $h(x)\equiv0$,
 while our Assumption \ref{assumpt:key}  does not imply the strong growth condition. Therefore in the smooth case Assumption \ref{assumpt:key} is strictly weaker than the SGC. Because, unlike  \cite{Schmidt-Roux,Vaswani-Bach-Schmidt}, we do not assume unbiased gradient estimators, we need to impose a condition of the rate of decay of expected noise, similar to the conditions on the noise bound in \cite{Schmidt-Roux-Bach}.  
Results in \cite{Vaswani-Bach-Schmidt} are also adapted to the case when strong growth condition is relaxed by adding 
another term. It is easy to see in \cite{Vaswani-Bach-Schmidt} that this term needs to have the same rate of decay as in this paper to attain accelerated convergence rates. In this sense our results are comparable to those in \cite{Vaswani-Bach-Schmidt}.

There is a variety of other related works on inexact proximal methods \cite{Bonettini-Prato-Rebegoldi,Bonettini-Rebegoldi-Ruggiero,Rebegoldi-Calatroni,Sun-Barrio-Jiang-Cheng}, however they investigate inexact solution in the proximal step, which is quite different from the inexact gradient setting in this paper.

\section{ISTA Step Search Algorithm Based on Stochastic First-order Oracle}\label{sec:ista}
In this section we introduce our stochastic first-order oracle and describe a variant of ISTA algorithm based on this oracle. 
This algorithm generates  {several stochastic properties}, that we analyze in this section. Based on these properties we are able to derive the bound on the expected iteration complexity of this algorithm. 
\subsection{Preliminary Notations}
Given a  point $y\in\R^n$ and a step-size parameter $\alpha$, we define the following approximation $Q_\alpha(x, y)$ of the objective function $F$: 
\begin{align*}
Q_\alpha(x, y) = f(y) +\nabla f(y)^\top (x -y)  + \frac{1}{2\alpha}\|x-y\|^2 + h(x).
\end{align*}
The minimizer of this function with respect to $x$, which we denote by $p_\alpha (y)$ is given by the prox-operator: 
\begin{align*}
    p_\alpha (y) := \arg\min_x Q_\alpha(x, y) = \text{prox}_{\alpha h} (y - \alpha \nabla f(y)).
\end{align*}
We let $D_{\alpha}(y)$ denote the gradient mapping of $F$ \cite{Nesterov}:
\begin{align}\label{eq:gradient_mapping}
D_{\alpha}(y) := \frac{1}{\alpha}\Big(y - \text{prox}_{\alpha h}(y - \alpha\nabla{f}(y))\Big)= \frac{1}{\alpha}\Big(y - p_\alpha (y)\Big), \text{ for } \alpha > 0.
\end{align}
Note that when $ h\equiv 0$, then $D_{\alpha}(y) \equiv \nabla{f}(y)$.
We also define the following approximating function based on an estimate  $g$ instead of  $\nabla f(y)$:
\begin{align*}
\tilde{Q}_\alpha(x, y) = f(y) + g^\top (x -y)  + \frac{1}{2\alpha}\|x-y\|^2 + h(x).
\end{align*}
Similarly we define
\begin{align*}
    \tilde{p}_\alpha (y) := \arg\min_x \tilde{Q}_\alpha(x, y) = \text{prox}_{\alpha h} (y - \alpha g).
\end{align*}

\subsection{Stochastic First-order Oracle and Step Search Algorithm}

 We now present  the first-order stochastic oracle  on  which  the methods in this paper rely.  This oracle is similar to the oracles used in \cite{Berahas-Cao-Scheinberg,Carter,Polyaka}. It is assumed to generate a stochastic vector, which is sufficiently close to the true gradient with at least a given probability, but otherwise can be arbitrary. The measure of closeness of the estimate to the true gradient varies by the setting. In \cite{Berahas-Cao-Scheinberg,Cartis-Scheinberg} for instance, the accuracy is proportional to the norm of either the true gradient  {or} the estimated gradient, which goes to zero as the method converges. Here, due to the presence of non-smooth function $h$,  the accuracy is essentially dictated by the current ``closeness'' to optimality as measured by the norm of the gradient mapping $D_{\alpha}(y)$.


{\bf Stochastic first-order oracle SFO$_{\constG, p}(y,\alpha)$.}	 
{\em Given a vector $y\in \R^n$ and a step-size parameter $\alpha > 0$, the oracle computes $G(y,\xi)$, a (random) estimate of the gradient $\nabla f (y)$, such that
\begin{align}\label{eq:first_order1}
\mathbb P_{\xi} \left\{\|G(y,\xi)-\nabla f(y)\|\leq \constG \cdot  \|D_{\alpha}(y)\|   \right\}
\geq p, 
\end{align}
where $\xi$ is a random variable whose distribution may depend on $y$ and $D_{\alpha}(y)$ is defined in \eqref{eq:gradient_mapping}. The values of the constants  $\constG$ and $p$ are intrinsic to the oracle, and satisfy  {$\constG \leq \frac{1}{3}$} and $p > \frac{1}{2}$.
In summary, the input to the oracle is $(y,\alpha)$,  the output is $G(y,\xi)$. 
(We will omit the dependence on $\xi$ for brevity) 
}

  {
Let us discuss  our assumption \eqref{eq:first_order1} on the gradient error and relate it to similar assumptions in the literature. 
 When $ h\equiv 0$ then the right hand side in \eqref{eq:first_order1} reduces to 
 $ \constG\|\nabla f(y)\|$, which is the same assumption  as  used in \cite{Berahas-Cao-Scheinberg} (with slightly different condition on $\constG$, which, in turn, is a  stochastic relaxation
 of the well known {\em norm condition} on the gradient error \cite{Berahas-Cao-Scheinberg,Carter,Polyak}. It also can be seen as related to  the {\em strong growth condition}  often used in the literature for stochastic gradient, (see e.g. \cite{Vaswani-Bach-Schmidt}) when  $ h\equiv 0$. The strong growth condition is stronger than \eqref{eq:first_order1}  in the sense that the bound on the gradient error is imposed in expectation, not simply with a given probability. On the other hand, the general strong growth condition is also weaker, becasue $\constG$ is allowed to be arbitrarily large. When gradient descent type methods are analyzed under strong growth condition, they are applied with a fixed and sufficiently small step size. The step size is inversely proportional to $\constG$ and thus the convergence rate scales with $\constG$. We require $\constG\leq \frac{1}{3}$ here because of the use of the sufficient decrease condition for line search (similarly as is done in \cite{Berahas-Cao-Scheinberg}). Because the sufficient decrease condition uses $\|G(y,\xi)\|$, it is necessary to ensure that this norm is sufficiently large (at least with probability $p$) when $\|\nabla f(y)\|$ is large. It can be observed that if strong growth condition holds with a large constant $\constG$ then our condition can always be satisfied by averaging a  sufficiently large number of stochastic gradient estimates.  }
 
  {Strong growth condition holds naturally in some specfic cases of stochastic gradients, such as 
 randomized finite difference gradient estimation \cite{Scheinberg} and in machine learning under {\em interpolation condition} \cite{Mishkin}. 
 When it does not hold one may want to impose it by averaging a  sufficiently large number of estimates. However, choosing this number may not be possible without the knowledge of  $\|\nabla f(y)\|$. Various alternative stochastic first order oracle conditions have been derived in the literature \cite{byrd2012sample,Cartis-Scheinberg}. They usually only provide heuristic methodology but can be more practical.  }
 
  {
Now, let us turn to the case when $ h\not \equiv 0$. The right hand size now replaces $\|\nabla f(y)\|$ with the norm of the gradient mapping. Deterministic version of this condition has been 
used in  \cite{Baraldi-Kouri} for the analysis of a  proximal inexact trust-region algorithm. A stochastic version imposed in expectation was used in \cite{Beiser-et-al} and an alternative, that is meant to be more practical, is suggested in \cite{Xie-et-al}.  Further variants for general constrained optimization are proposed in \cite{Bollapragada-et-al}. Similarly, in our case, we can consider using 
 $\tilde{D}_\alpha (y)$ in the right hand side instead of  $\tilde{D}_\alpha (y)$, however ensuring 
 this modified condition at least in principle is as hard as \eqref{eq:first_order1}. 
  Implementability of these conditions in various settings and  their practical performance  are still an active area of research that did not yet issue clear winners. Thus, here, we chose to impose the condition in its most natural form as an extension of the smooth case. 
}

We now present what we call the ISTA Step Search Algorithm based on the stochastic first-order oracle SFO$_{\constG, p}(y,\alpha)$. This algorithm is similar to ISTA with backtracking  in \cite{Beck-Teboulle},  except that instead of the gradient $\nabla f(y)$ the algorithm only has access to its stochastic estimate $G(y,\xi)$. Unlike ISTA, the true gradient is used and thus does not change unless the iterate is changed,   in this method  the gradient estimate is recomputed at each iteration whether the iterate is changed or not. Specifically, if the sufficient reduction condition is satisfied, then the step is accepted and the step-size  parameter is increased (successful iteration), otherwise the step is rejected and the step-size  parameter is decreased (unsuccessful iteration). In either case new iteration starts and the gradient estimate is recomputed.  Note that in the  ISTA algorithm described in \cite{Beck-Teboulle} step-size parameter is never increased. However, the analysis of ISTA easily holds when the step-size parameter  is increased on successful iterations. For the stochastic version this increase is essential, therefore we use it here.  The Algorithm \ref{alg:ista_bktr_line_search} makes reference to two iterate sequences $\{y_{k}\}$ and $\{x_{k}\}$ with $y_{k}=x_{k-1}$, for all $k$. One of the sequences is thus redundant for the ISTA algorithm, but we used them for consistent notation and analysis with the FISTA algorithm, where $y_{k}\neq x_{k-1}$, and both of  these two sequences are needed. 
\begin{algorithm}[ht]
\caption{Stochastic ISTA Step Search Algorithm}
\label{alg:ista_bktr_line_search}
\begin{algorithmic}[1]
\STATE {Initialization:} 
Choose an initial point $x_0\in \mathbb{R}^d$. Set $0<\gamma<1$,  with $\alpha_1>0$. Let oracle SFO$_{\constG, p}$ be given with  {$0 \leq \constG \leq \frac{1}{3}$} and $\frac{1}{2} < p \leq 1 $.
\FOR{$k=1,2,\cdots,$}
    \STATE Set  $y_k = x_{k-1}$.
    \STATE Compute the gradient estimate $g_k=G(y_k,\xi_k)$ using the stochastic first-order oracle SFO$_{\constG, p}(y_k,\alpha_k)$.
    Compute the reference point $\tilde{p}_{\alpha_k} (y_k) =\text{prox}_{\alpha_k h} (y_k- \alpha_k g_k)$.
    \STATE Check the sufficient decrease condition $F(\tilde{p}_{\alpha_k} (y_k)) \leq \tilde{Q}_{\alpha_k}(\tilde{p}_{\alpha_k} (y_k), y_k)$.
        \STATE \hspace{1.25em}If unsuccessful: 
            \STATE \hspace{2.5em}Set $x_{k} = x_{k-1}$ and  $\alpha_{k+1} = \gamma\alpha_k$.
        \STATE \hspace{1.25em}If successful: 
            \STATE \hspace{2.5em}Set $x_k = \tilde{p}_{\alpha_k} (y_k)$ and $\alpha_{k+1} = \tfrac{\alpha_k}{\gamma}$.
\ENDFOR
\end{algorithmic}
\end{algorithm}

\subsection{Stochastic Process and Stochastic Oracle Assumption}\label{sec:stoc_proc}

\textit{Filtration.}
Algorithm \ref{alg:ista_bktr_line_search} generates a stochastic process $\{Y_k, X_k, G_k, \Stepsize{k}\}$, where $G_k$ denotes the random gradient estimate computed by the oracle SFO$_{\constG, p}(Y_k,\Stepsize{k})$ at iteration $k$, $Y_k$ and $X_k$ are the random iterates and  $\Stepsize{k}$ is the random step-size  parameter. 
We use $g_k$, $y_k$, $x_k$ and $\alpha_k$  to denotes realizations of these quantities for a particular sample path. We omit explicitly stating the dependence on the sample, for brevity. 
Let $\pastone$ denote the $\sigma$-algebra generated by the random variables $G_0, \hdots, G_{k-1}$. We note that by the nature of the algorithm, distribution of $G_k$ is determined by $Y_k$, while  $X_k$,  $\Stepsize{k+1}$, and, consequently,  $Y_{k+1}$ are  deterministic conditioned on $G_k$. Therefore  process $\{G_k\}$ captures the randomness of the stochastic process $\{Y_k, X_k, G_k,  \Stepsize{k}\}$, in other words,  $\{\pastone\}_{k\ge 1}$ is a filtration. 


\textit{The hitting time $N_\epsilon$ to reach $\epsilon$ accurate solution.} 
Given a target  accuracy $\epsilon>0$, we aim to derive a bound on the expected number of iterations after which
the algorithm generates the iterate that satisfies $\{F(x_k)-F^*\leq \epsilon\}$. We denote the optimality gap $v_k = F(x_k)-F^*$ for every iteration $k$.
Let  $N_\epsilon=\min\{k=1,2, \ldots\ : F(X_k)-F^*\leq \epsilon\}$, that is $N_\epsilon$ is the first iteration for which the event $\{F(x_k)-F^*\leq \epsilon\}$ occurs.  $N_\epsilon$ is a stopping time (hitting time) with respect to the filtration $\{\mathcal{F}_k\}_{k\ge 1}$   as it is defined by the realization of $X_k$.

We find the following indicator variables useful. 

\begin{definition}[Successful Iteration]\label{def:success} 
Let $\iterSucc=\mathbbm{1}\{{\rm Iteration\ } k\ {\rm is\ successful } \}$.
\end{definition}

\begin{definition}[True Iteration]\label{def:good_model}
We say iteration $k$ is {\em true} if 
\begin{align*}
    \|g_k-\nabla f(y_k)\|\leq \constG  \|D_{\alpha_k}(y_k)\|, \  {\constG\leq \frac{1}{3}}.
\end{align*}

In other words, the first-order oracle delivers a sufficiently accurate estimate. 
We use an indicator variable $\iterTrue$ to indicate that iteration $k$ is true. 
\end{definition}

Our key assumption on the oracle is that  it is implementable with  {$\constG \leq \frac{1}{3}$} and $p > \frac{1}{2}$.

\begin{assumption}[Probability of True Iteration]\label{assumpt:key}
We assume that the first-order oracle  SFO$_{\constG, p}$ is implemented in Algorithm \ref{alg:ista_bktr_line_search} with  {$\constG \leq \frac{1}{3}$} and $p > \frac{1}{2}$ for every input $(y_k,\alpha_k)$. 
Then, by the properties of  the oracle and the filtration $\pastone$ we have, 
\begin{align*} 
\Prob(\iterTrue| \pastone) \ge p \text{ for all } k \geq 0.
\end{align*}
\end{assumption}
\subsection{Key Elements of the Analysis}\label{subsec:successful}
Now we are ready to present our analysis, starting with a property of the algorithms that holds for each realization. 
Our first lemma shows that a small step-size parameter  and a true iteration lead to a successful step. This is an equivalent of Lemma 3.1 in \cite{Cartis-Scheinberg}, extended to the proximal step. 

\begin{lemma}
\label{lem:good_est_good_model_successful} 
Suppose that iteration $k$ is true, i.e. 
\begin{align*}
    \|g_k-\nabla f(y_k)\|\leq \constG  \|D_{\alpha_k}(y_k)\|,
\end{align*}
with  {$\constG \leq \frac{1}{3}$} and $D_{\alpha_k}(y_k)$ is the gradient mapping of $F$ at $y_k$. If 
\begin{align*}
\stepsize{k}{} \le \stepsizebar := \frac{1}{L}\left( 1 - \frac{2\constG}{1-\constG}\right) ,
\end{align*}
then the $k$-th step is successful. 

\end{lemma} 

\begin{proof}

From the $L$-smoothness of $f$ we have
\begin{align*}
    f(\tilde{p}_{\alpha_k} (y_k))  &\le f(y_k) + \nabla f(y_k)^\top(\tilde{p}_{\alpha_k} (y_k) - y_k) + \frac{L}{2}\norm{\tilde{p}_{\alpha_k} (y_k) - y_k}^2 \nonumber\\
     &= f(y_k) +  g_k^\top (\tilde{p}_{\alpha_k} (y_k) - y_k) + [\nabla f(y_k)-g_k]^\top (\tilde{p}_{\alpha_k} (y_k) - y_k)  \nonumber\\
     &\quad+ \frac{1}{2\alpha_k} \norm{\tilde{p}_{\alpha_k} (y_k) - y_k}^2  + \left(\frac{L}{2}- \frac{1}{2\alpha_k}\right)\norm{\tilde{p}_{\alpha_k} (y_k) - y_k}^2 .
\end{align*}
Adding the non-smooth term $ h(\tilde{p}_{\alpha_k} (y_k))$ we have
\begin{align*}
    &F (\tilde{p}_{\alpha_k} (y_k))  \nonumber\\
     &\le f(y_k) +  g_k^\top (\tilde{p}_{\alpha_k} (y_k) - y_k) + \frac{1}{2\alpha_k} \norm{\tilde{p}_{\alpha_k} (y_k) - y_k}^2 + h(\tilde{p}_{\alpha_k} (y_k))  \nonumber\\
     &\quad + [\nabla f(y_k)-g_k]^\top (\tilde{p}_{\alpha_k} (y_k) - y_k) + \left(\frac{L}{2}- \frac{1}{2\alpha_k}\right)\norm{\tilde{p}_{\alpha_k} (y_k) - y_k}^2 
     \nonumber\\
     &= \tilde{Q}_{\alpha_k}(\tilde{p}_{\alpha_k} (y_k), y_k) + [\nabla f(y_k)-g_k]^\top (\tilde{p}_{\alpha_k} (y_k) - y_k) + \left(\frac{L}{2}- \frac{1}{2\alpha_k}\right)\norm{\tilde{p}_{\alpha_k} (y_k) - y_k}^2 .
\end{align*}
To show that the iteration is successful i.e. $F(\tilde{p}_{\alpha_k} (y_k)) \leq \tilde{Q}_{\alpha_k}(\tilde{p}_{\alpha_k} (y_k), y_k)$
we need to prove that 
\begin{align}\label{eq:02}
    [\nabla f(y_k)-g_k]^\top (\tilde{p}_{\alpha_k} (y_k) - y_k) \leq \left(\frac{1}{2\alpha_k} - \frac{L}{2}\right)\norm{\tilde{p}_{\alpha_k} (y_k) - y_k}^2 .
\end{align}
Note that the right hand side is non-negative since $\alpha_k \leq \frac{1}{L}$. 
Our definition of the true iteration leads to 
\begin{align}\label{eq:01}
    \|g_k-\nabla f(y_k)\|
    &\leq \constG \left \|\frac{1}{\alpha_k}\Big(y_k - p_{\alpha_k} (y_k)\Big)\right\|\nonumber  \\
    &\leq  \frac{\constG}{\alpha_k} \left \|y_k - \tilde{p}_{\alpha_k} (y_k)\right\| + \frac{\constG}{\alpha_k} \left \|\tilde{p}_{\alpha_k} (y_k) - p_{\alpha_k} (y_k)\right\|\nonumber  \\
    &=  \frac{\constG}{\alpha_k} \left \|y_k - \tilde{p}_{\alpha_k} (y_k)\right\| + \frac{\constG}{\alpha_k} \norm{\alpha_k\nabla f(y_k)-\alpha_k g_k },
\end{align}
where the last inequality follows  from the nonexpansiveness of the prox-operator $\text{prox}_{\alpha h}$ for  proper, closed, and convex $h(x)$, i.e.,
\begin{align*}
    \|\text{prox}_{\alpha h}(z_1) - \text{prox}_{\alpha h}(z_2)\| \leq \|z_1 - z_2\| \text{ for all } z_1, z_2\in\R^d.
\end{align*}
The bound \eqref{eq:01} is equivalent to
\begin{align*}
    (1-\constG)\norm{\nabla f(y_k)-g_k} \leq  \frac{\constG}{\alpha_k} \left \|y_k - \tilde{p}_{\alpha_k} (y_k)\right\|.
\end{align*}
Therefore we show \eqref{eq:02} by
\begin{align*}
    \norm{\nabla f(y_k)-g_k} &\leq  \frac{\constG}{1-\constG} \frac{1}{\alpha_k} \left \|y_k - \tilde{p}_{\alpha_k} (y_k)\right\|  \nonumber\\
    &\leq \left( \frac{1- \alpha_k  L}{2}\right)\frac{1}{\alpha_k} \left \|y_k - \tilde{p}_{\alpha_k} (y_k)\right\| = \left( \frac{1}{2\alpha_k} - \frac{L}{2}\right)\left \|y_k - \tilde{p}_{\alpha_k} (y_k)\right\|,
\end{align*}
where we use the condition that 
\begin{align*}
    \frac{\constG}{1-\constG} \leq \frac{1- \alpha_k  L}{2},
\end{align*}
which is the consequence of the small step-size parameter: 
\begin{align*}
    \alpha_k  \leq \frac{1}{L}\left( 1 - \frac{2\constG}{1-\constG}\right).
\end{align*}

~\qed \end{proof}

The next part of the analysis follows similar logic to that in \cite{Cartis-Scheinberg} and subsequent works \cite{Berahas-Cao-Scheinberg,Blanchet-Cartis-Menickelly-Scheinberg,Paquette-Scheinberg}. Specifically, the algorithm makes progress when the iterations are true, successful and the step-size parameter $\alpha_k$ is not too small.  Let us call these iterations {\em good}. The key idea is to bound 
 the total expected number of iterations in terms of the  expected number of {\em good} iterations.  
Let us define
$N_G = \sum_{k=1}^{N_\epsilon-1} \mathbbm{1}\{ \Alpha_k \ge \stepsizebar \} \cdot  \mathbbm{1}\{{\rm Iteration\ } k\ {\rm is\ true\ and \ successful} \} $ be the number of {\em good} iterations prior to the stopping time $N_\epsilon$.
 
\begin{theorem}[Bounding $\EE (N_\epsilon)$ based on $\EE(N_G)$]\label{th:framework}
Let $\stepsizebar$ be the value stated in Lemma \ref{lem:good_est_good_model_successful}.
 Under Assumption \ref{assumpt:key} we have 
\begin{align*}
\EE (N_\epsilon -1) \leq \frac{2p}{(2p-1)^2}\left (2\EE(N_G)+ \log_\gamma\left(\frac{\stepsizebar}{\alpha_1}\right)\right).
\end{align*}
\end{theorem}
\begin{proof}
By the update rule of $\alpha_k$ in Algorithm \ref{alg:ista_bktr_line_search}, the stochastic process $\{\Alpha_k\}_{k\geq 1}$ follows: 
\begin{align*}
\Alpha_{1}&= \alpha_1, \quad 
\Alpha_{k+1}=\left \{ \begin{array}{ll} 
\gamma^{-1} \Alpha_k & {\rm if\  }  \iterSucc=1,  \\
 \gamma \Alpha_k &   {\rm if\  }  \iterSucc=0 .
\end{array}\right . 
\end{align*}

Let us now define the following indicator random variables: 
\begin{align*}
\iterLarge=\mathbbm{1}\{ \Alpha_k >\stepsizebar\}, \\
\iterLargebar= \mathbbm{1}\{ \Alpha_k \geq \stepsizebar\}.
\end{align*}
Clearly 
\begin{align*}
\EE (N_\epsilon-1) = \EE\left (\sum_{k=1}^{N_\epsilon-1} \iterLarge\right ) + \EE \left (\sum_{k=1}^{N_\epsilon-1} (1-\iterLarge)\right ).
\end{align*}

From Lemma 2.3 in \cite{Cartis-Scheinberg}, under Assumption \ref{assumpt:key} and using the fact that all true iterations are successful when $\alpha_k \le \stepsizebar$, we have: 
\begin{align*}
\EE\left (\sum_{k=1}^{N_\epsilon-1} (1-\iterLarge)\right)\leq \frac{1}{2p}\EE(N_\epsilon -1 ).
\end{align*}

We define 
\begin{itemize}
\item $N_{FS}=\sum_{k=1}^{N_\epsilon-1} \iterLargebar (1-\iterTrue)\iterSucc$, which is the number of false successful iterations with $\Alpha_k\geq \stepsizebar$.
\item $N_F=\sum_{k=1}^{N_\epsilon-1} \iterLargebar (1-\iterTrue)$, which is the number of false  iterations with $\Alpha_k\geq \stepsizebar$.
\item $N_G=\sum_{k=1}^{N_\epsilon-1} \iterLargebar \iterTrue \iterSucc$, which is the number of true successful iterations with $\Alpha_k\geq \stepsizebar$.
\item $N_T=\sum_{k=1}^{N_\epsilon-1} \iterLargebar \iterTrue$, which is the number of true iterations with $\Alpha_k\geq \stepsizebar$.
\item $N_U=\sum_{k=1}^{N_\epsilon-1} \iterLarge  (1-\iterSucc)$, which is the number of unsuccessful iterations with $\Alpha_k> \stepsizebar$.
\end{itemize}

From equations (11) and (12) of Lemma 2.7 in \cite{Cartis-Scheinberg} we have
\begin{align}\label{eq:binom}
\EE(N_{FS})\leq \EE(N_{F})\leq\frac {1-p}{p}[\EE(N_G)+\EE(N_U)],
\end{align}
and
\begin{align}\label{eq:boundM3}
\EE(N_U)\leq \EE(N_{FS})+\EE(N_G) + \log_\gamma(\stepsizebar/\alpha_1).
\end{align}
Plugging \eqref{eq:boundM3} into  \eqref{eq:binom}
we obtain
\begin{align*}
\EE(N_{FS})\leq \frac {1-p}{p}\left[\EE(N_{FS})+2\EE(N_G) + \log_\gamma(\stepsizebar/\alpha_1)\right],
\end{align*}
and, hence,
\begin{align*}
\frac {2p-1}{p}\EE(N_{FS})\leq \frac {1-p}{p}\left[ 2\EE(N_G) + \log_\gamma\left(\frac{\stepsizebar}{\alpha_1}\right)\right].
\end{align*}
This finally implies 
\begin{align}\label{eq:boundN1}
\EE(N_{FS})\leq \frac {1-p}{2p-1}\left [2\EE(N_G)+ \log_\gamma\left(\frac{\stepsizebar}{\alpha_1}\right)\right ].
\end{align}
Now we can bound the expected total  number of  iterations when $\alpha_k>\stepsizebar$, using 
\eqref{eq:boundM3} and \eqref{eq:boundN1}  
and adding the terms to obtain the result of the lemma, namely,
\begin{align*}
\EE(N_F+N_T)&\leq \EE(N_F+N_U+N_G)\leq \frac{1}{p}\EE(N_U+N_G)\nonumber\\
&\leq \frac{1}{2p-1}\left (2\EE(N_G)+ \log_\gamma\left(\frac{\stepsizebar}{\alpha_1}\right)\right).
\end{align*}
Next, from Lemma 2.7 in \cite{Cartis-Scheinberg},\footnote{Final statement of the Lemma 2.7 in \cite{Cartis-Scheinberg}  replaces $\EE(N_G)$ with its bound, but here we
are yet to derive such a  bound, therefore we explicitly use $\EE(N_G)$.}
\begin{align*}
\EE\left (\sum_{k=1}^{N_\epsilon-1}\iterLarge \right )\leq \frac{1}{2p-1}\left (2\EE(N_G)+ \log_\gamma\left(\frac{\stepsizebar}{\alpha_1}\right)\right). 
\end{align*}
Hence, using the two previous results we have 
\begin{align*}
\EE (N_\epsilon -1)\leq \frac{1}{2p} \EE (N_\epsilon -1) +   \frac{1}{2p-1} \left (2\EE(N_G)+ \log_\gamma\left(\frac{\stepsizebar}{\alpha_1}\right)\right).
\end{align*}
The result of the theorem follows. 
~\qed \end{proof}

\subsection{Analysis for Stochastic ISTA Step Search}\label{subsec:analysis1}

We now derive the bound on $\EE(N_G)$ based on the progress made on each iteration. 
We first state the following lemma that holds for each realization of Algorithm \ref{alg:ista_bktr_line_search}. 
Without the stochastic term, this lemma is equivalent to the bound (3.6) for the ISTA backtracking algorithm in \cite{Beck-Teboulle}.
\begin{lemma}[Bounds for a successful step]\label{lem:successful_step_bound1}
If $k$ is a successful iteration then we have:
\begin{align*}
\|x_{k} - x^*\|^2 - \|x_{k-1} - x^* \|^2 \leq  - 2\alpha_k (F(x_k) - F(x^*) )  + \lambda_k \cdot \|x_{k-1} -x^*\|,
\end{align*}
where $\lambda_k = \| 2\alpha_k (\nabla f(x_{k-1})-g_k) \| $.
\end{lemma}

\begin{proof}
Since iteration $k$ is successful,  the sufficient decrease condition holds: 
\begin{align*}
    F(\tilde{p}_{\alpha_k} (x_{k-1})) &\leq \tilde{Q}_{\alpha_k}(\tilde{p}_{\alpha_k} (x_{k-1}), x_{k-1}) \nonumber\\
    &= f(x_{k-1}) +(x_k -x_{k-1})^\top g_k + \frac{1}{2\alpha_k}\|x_k-x_{k-1}\|^2 + h(x_k).
\end{align*}
By the  convexity of $f$ and $h$, applied at $x_{k-1}$ and $x_k$ respectively, for any $x$,  
 \begin{align*}
f(x) &\geq f(x_{k-1}) + \nabla f(x_{k-1})^\top (x-x_{k-1}),
\nonumber\\
h(x) &\geq h(x_k) + \gamma(x_{k-1})^\top (x-x_k),
\end{align*}
where the second bound holds for every subgradient $\gamma(x_{k-1})$ of $h$ at $x_{k}$. 
Summing those inequality we have 
\begin{align*}
f(x) +h(x) \geq f(x_{k-1}) + \nabla f(x_{k-1})^\top (x-x_{k-1})+ h(x_k) + \gamma(x_{k-1})^\top (x-x_k).
\end{align*}
Combining this with the sufficient decrease condition we get: 
\begin{align*}
    &F(x) - F(x_k) \nonumber\\&\geq \nabla f(x_{k-1})^\top (x-x_{k-1}) + \gamma(x_{k-1})^\top (x-x_k) - (x_k -x_{k-1})^\top g_k - \frac{1}{2\alpha_k}\|x_k-x_{k-1}\|^2\nonumber\\
    &= (\nabla f(x_{k-1})-g_k) ^\top (x-x_{k-1}) + (\gamma(x_{k-1})+g_k)^\top (x-x_k) - \frac{1}{2\alpha_k}\|x_k-x_{k-1}\|^2.
\end{align*}
Using the fact that $x_k = \tilde{p}_{\alpha_k} (x_{k-1}) =\arg\min_x \tilde{Q}_{\alpha_k}(x, x_{k-1})$, we get that $ x_{k-1} - x_k= \alpha_k[\gamma(x_{k-1})+g_k]$ for some subgradient $\gamma(x_{k-1})$ of $h$ at $x_{k}$. Therefore
\begin{align*}
    &F(x) - F(x_k) \nonumber\\
    &\geq (\nabla f(x_{k-1})-g_k) ^\top (x-x_{k-1}) + \frac{1}{\alpha_k}(x_{k-1} - x_k)^\top (x-x_k) - \frac{1}{2\alpha_k}\|x_k-x_{k-1}\|^2\nonumber\\
    &= (\nabla f(x_{k-1})-g_k) ^\top (x-x_{k-1}) + \frac{1}{\alpha_k}(x_{k-1} - x_k)^\top (x-x_{k-1} + x_{k-1} -x_k) \nonumber\\
    & \quad - \frac{1}{2\alpha_k}\|x_k-x_{k-1}\|^2\nonumber\\
    &= (\nabla f(x_{k-1})-g_k) ^\top (x-x_{k-1}) + \frac{1}{\alpha_k}(x_{k-1} - x_k)^\top (x-x_{k-1}) + \frac{1}{2\alpha_k}\|x_k-x_{k-1}\|^2\nonumber\\
    &= (\nabla f(x_{k-1})-g_k) ^\top (x-x_{k-1}) + \frac{1}{2\alpha_k} \left( \|x-x_k \|^2-\|x-x_{k-1} \|^2\right),
\end{align*}
and we get the bound for $x=x^*$: 
\begin{align*}
    &2\alpha_k (F(x_k) - F(x^*) ) - 2\alpha_k (\nabla f(x_{k-1})-g_k) ^\top (x_{k-1} -x^*)
    \leq \|x_{k-1} - x^*\|^2 - \|x_k - x^* \|^2.
\end{align*}
Rearranging the terms, we get the statement of Lemma \ref{lem:successful_step_bound1}:
\begin{align*}
    \|x_k - x^* \|^2 - \|x_{k-1} - x^*\|^2 &\leq -2\alpha_k (F(x_k) - F(x^*) ) + 2\alpha_k (\nabla f(x_{k-1})-g_k) ^\top (x_{k-1} -x^*)\nonumber\\
    &\leq -2\alpha_k (F(x_k) - F(x^*) ) + \lambda_k \cdot\|x_{k-1} -x^*\|,
\end{align*}
where $\lambda_k = \|2\alpha_k (\nabla f(x_{k-1})-g_k)\|$.

~\qed \end{proof}

The term $\|2\alpha_k (\nabla f(x_{k-1})-g_k)\| \cdot\|x_{k-1} -x^*\| $ in the above lemma bounds the worst case error in each iteration. In the next lemma we adapt the argument used in \cite{Schmidt-Roux-Bach} to show how these errors accumulate in the total bound on the optimality. 
\begin{lemma}\label{lem:after_K1}
    After $K$ iterations of the algorithm, the following bound holds: 
\begin{align*}
\sum_{k\in S, k \leq K} 2\alpha_k (F(x_k) - F(x^*) )
    &\leq 2\|x_0 - x^* \|^2 + \sum_{k=1}^K \lambda_k \cdot \left( \sum_{i=1}^{k-1} \lambda_i \right)+ 2\sum_{k=1}^K \lambda_k^2,
\end{align*}
where $S$ is the set of successful iterations.
\end{lemma}
\begin{proof}
For every $K$, from Lemma \ref{lem:successful_step_bound1} we have
\begin{align*}
    \|x_K - x^* \|^2 - \|x_0 - x^* \|^2 
    &= \sum_{k=1}^{K} \left(\|x_{k} - x^* \|^2 -\|x_{k-1} - x^* \|^2 \right) \nonumber\\
    &\leq \sum_{k\in S, k \leq K} \left[\lambda_k \cdot \|x_{k-1} -x^*\| - 2\alpha_k (F(x_k) - F(x^*) ) \right],
\end{align*}
where $S$ is the set of successful iterations. 
Hence we arrive at our key bound
\begin{align}\label{eq:thm_key1}
    &\|x_K - x^* \|^2 - \|x_0 - x^* \|^2 + \sum_{k\in S, k \leq K} 2\alpha_k (F(x_k) - F(x^*) )\leq  \sum_{k=1}^K \lambda_k  \cdot \| x_{k-1} -x^* \|.
\end{align}
Now we follow a similar argument as \cite{Schmidt-Roux-Bach} with two steps. First, we use the key bound \eqref{eq:thm_key1} to bound the distance to optimal point $u_{k-1} = \| x_{k-1} -x^* \|$ in terms of the initial distance $u_{0} = \| x_{0} -x^* \|$ and the noise terms $\lambda_k$. Then we use this bound on $\| x_{k-1} -x^* \|$ to substitute into the last sum of the key bound \eqref{eq:thm_key1} and obtain the final derivation in terms of the noise $\lambda_k$.\\
Dropping the positive term $\sum_{k\in S, k \leq K} 2\alpha_k (F(x_k) - F(x^*) )$ we get
\begin{align*}
    \|x_K - x^* \|^2 
    &\leq \|x_0 - x^* \|^2 + \sum_{k=1}^K \lambda_k \cdot \| x_{k-1} -x^* \| \nonumber\\
    &= \|x_0 - x^* \|^2 + \lambda_1 \cdot \| x_{0} -x^* \| + \sum_{k=1}^K \lambda_{k+1} \cdot\| x_{k} -x^* \|.
\end{align*}
Letting $u_k = \| x_{k} -x^* \|$, this is equivalent to 
\begin{align*}
    u_K^2 = \|x_K - x^* \|^2  \leq Z + \sum_{k=1}^K\lambda_{k+1} u_k.
\end{align*}
where $Z = \|x_0 - x^* \|^2 + \lambda_1 \cdot \| x_{0} -x^* \| = u_0^2 + \lambda_1 u_0$. 
Using this inequality and applying Lemma 1 in \cite{Schmidt-Roux-Bach} for the non-negative sequence $\{u_K\}$ have
\begin{align*}
     u_{k} \leq \frac{1}{2} \sum_{i=1}^{k} \lambda_{i+1} + \left(Z + \left(\frac{1}{2} \sum_{i=1}^{k} \lambda_{i+1} \right)^2\right)^{1/2}
    \leq  \sum_{i=1}^{k} \lambda_{i+1} + Z^{1/2}  \leq  \sum_{i=1}^{k+1} \lambda_{i} + u_0 ,
\end{align*}
where the last line follows from the fact that $Z^{1/2} = \left(u_0^2 + \lambda_1  \cdot u_0 \right)^{1/2} \leq u_0 + \lambda_1.$
\\
Putting the bound for $u_k$ into the key bound \eqref{eq:thm_key1} we have
\allowdisplaybreaks
\begin{align*}
    &\|x_K - x^* \|^2 - \|x_0 - x^* \|^2 + \sum_{k\in S, k \leq K} 2\alpha_k (F(x_k) - F(x^*) )\nonumber\\
    &\leq \sum_{k=1}^K \lambda_k \cdot \| x_{k-1} -x^* \|\leq \sum_{k=1}^K \lambda_k \cdot \left( \sum_{i=1}^{k} \lambda_i + \|x_0 - x^* \|\right)\nonumber\\
    &= \sum_{k=1}^K \lambda_k \cdot \left( \sum_{i=1}^{k} \lambda_i \right) + \|x_0 - x^* \| \cdot \sum_{k=1}^K \lambda_k\nonumber\\
    &\leq  \sum_{k=1}^K \lambda_k \cdot \left( \sum_{i=1}^{k-1} \lambda_i \right)+ \sum_{k=1}^K \lambda_k^2 + \|x_0 - x^* \|^2 + \sum_{k=1}^K \lambda_k^2 \nonumber\\
    &= \|x_0 - x^* \|^2 + \sum_{k=1}^K \lambda_k \cdot \left( \sum_{i=1}^{k-1} \lambda_i \right)+ 2\sum_{k=1}^K \lambda_k^2,
\end{align*}
and the statement of Lemma \ref{lem:after_K1} follows. 
~\qed
\end{proof}

In \cite{Schmidt-Roux-Bach} the error term is assumed to be deterministically bounded in such a way that 
the total accumulation is bounded by a constant (not dependent on the total number of iterations). Here we can relax this assumption somewhat, since we only need the expected error terms to be summable. Therefore we make the following additional assumption on the output of oracle SFO, which can be seen as the relaxation of assumptions in  \cite{Schmidt-Roux-Bach}. 
\begin{assumption}\label{assumpt:ista_noise}
At every iteration $k$,  $g_k=g(y_k,\xi_k)$ computed using the stochastic first-order oracle SFO$_{\constG, p}(y_k,\alpha_k)$ satisfies  
\begin{align}\label{eq:ista_noise}
    &\EE_{\xi_k} \left[ \| (\nabla f(y_{k} )- G_{k} )\|^2 | \pastone\right] \leq \frac{1}{\alpha_{k}^2\cdot k^{2+ \beta}},
\end{align}
for some fixed $\beta>0$. 
\end{assumption}

Note that essentially this assumption requires the error of the gradient estimate to decay a little bit faster than $\frac{1}{k}$. If the step-size parameter happens to be small, then more gradient error can be tolerated. We discuss this assumption in more detail at the end of the section. 

\begin{theorem}[Upper bound of the number of good steps]\label{th:bound_expectation_N_TS1} Under Assumption \ref{assumpt:convex_smooth}, \ref{assumpt:bounded_below}  ,
and  Assumption \ref{assumpt:ista_noise}, we have 
 \begin{align*}
\EE [N_G] 
\leq \frac{1}{2\stepsizebar\epsilon} \left [2\|x_0 - x^* \|^2 + \frac{4(\beta + 2)^2}{\beta^2} + \frac{8(\beta+2)}{\beta+1} 
 \right]. 
\end{align*}

\end{theorem}
\begin{proof}
We recall the definition of the noise term $\lambda_k =  \|2\alpha_{k} (\nabla f(y_{k} )- g_{k} )\| $. Let $\Lambda_k =  \|2\Alpha_{k} (\nabla f(Y_{k} )- G_{k} )\| $, thus $\Lambda_k$ is a random variable with realizations  $\lambda_k$. 
Since the statement of Lemma \ref{lem:after_K1} holds for every realizations of Algorithm \ref{alg:ista_bktr_line_search}, for $K = N_\epsilon -1$ we have:  
\begin{align*}
\sum_{k\in S, k \leq N_\epsilon -1} 2\Alpha_k (F(X_k) - F(x^*) )
    &\leq 2\|x_0 - x^* \|^2 + \sum_{k=1}^{N_\epsilon -1} \Lambda_k \cdot \left( \sum_{i=1}^{k-1} \Lambda_i \right)+ 2\sum_{k=1}^{N_\epsilon -1} \Lambda_k^2,
\end{align*}
where $S$ is the set of successful iterations.
Recall that the stopping time $N_\epsilon = \min\{k = 1,2, \ldots\,:  F(X_k) - F^*  \leq \epsilon\}$. Hence for every $k \leq N_\epsilon-1$, we have that $ F(X_k) - F^*> \epsilon$ and
\begin{align*}
    \epsilon \cdot \sum_{k\in S, k \leq N_\epsilon -1} 2\Alpha_k  &<  \sum_{k\in S, k \leq N_\epsilon -1} 2\Alpha_k (F(X_k) - F(x^*) )\nonumber\\
    &\leq 2\|x_0 - x^* \|^2 + \sum_{k=1}^{N_\epsilon} \Lambda_k \cdot \left( \sum_{i=1}^{k-1} \Lambda_i \right)+ 2\sum_{k=1}^{N_\epsilon} \Lambda_k^2.
\end{align*}
Let $\mathcal{N}_G$ be the set of \textit{good} iterations  i.e. true and successful iterations $k$ with $\alpha_k\geq \stepsizebar, k \leq N_\epsilon -1 $, then $N_G$ is the cardinality of that set. We have 
\begin{align*} 
\sum_{k\in S, k \leq N_\epsilon -1} 2\Alpha_k \geq\sum_{k\in \mathcal{N}_G, k \le N_\epsilon -1 }  2\Alpha_k \geq N_G \cdot 2\stepsizebar = 2 N_G \stepsizebar.
\end{align*}
Therefore 
\begin{align*}
    2 N_G \stepsizebar \cdot \epsilon  
    &\leq 2\|x_0 - x^* \|^2 + \sum_{k=1}^{N_\epsilon} \Lambda_k \cdot \left( \sum_{i=1}^{k-1} \Lambda_i \right)+ 2\sum_{k=1}^{N_\epsilon} \Lambda_k^2,
\end{align*}
and 
\begin{align*}
    N_G   
    &\leq \frac{2\|x_0 - x^* \|^2 + \sum_{k=1}^{N_\epsilon} \Lambda_k \cdot \left( \sum_{i=1}^{k-1} \Lambda_i \right)+ 2\sum_{k=1}^{N_\epsilon} \Lambda_k^2}{2 \stepsizebar \epsilon}.
\end{align*}
Taking expectation we have 
\begin{align}\label{eq:bound1}
\EE [N_G ] 
 &\leq \frac{1}{2\stepsizebar\epsilon} \left [2\|x_0 - x^* \|^2 
+ \EE \left[ \sum_{k=1}^{N_\epsilon} \Lambda_{k} \cdot \left(\sum_{i=1}^{k-1} \Lambda_{i} \right)  \right] + 2\EE \left[\sum_{k=1}^{N_\epsilon}  \Lambda_{k} ^2  \right] 
 \right]  .
\end{align}
In order to bound \eqref{eq:bound1}, we derive the expected values of $\Lambda_k^2$, $\Lambda_k$ and $\Lambda_k \Lambda_i$ for $1\leq i, k \leq N_\epsilon$
using Assumption \ref{assumpt:ista_noise}. From \eqref{eq:ista_noise} we have: 
\begin{align}\label{eq:mu_k_squared1}
    \EE \left[\Lambda_{k}^2| \pastone\right] = \EE \left[4\alpha_{k}^2 \| (\nabla f(y_{k} )- G_{k} )\|^2 | \pastone\right] \leq \frac{4}{k^{2+ \beta}},
\end{align}
which leads to 
\begin{align*}
    \left(\EE \left[\Lambda_{k}| \pastone\right]\right)^2 \leq \EE \left[\Lambda_{k}^2| \pastone\right] \leq \frac{4}{k^{2+ \beta}}.
\end{align*}
Therefore we have the bound for $\EE \left[\Lambda_{k}| \pastone\right]$: 
\begin{align}\label{eq:mu_k1}
    \EE \left[\Lambda_{k}| \pastone\right] \leq \frac{2}{k^{1+ 0.5\beta}}.
\end{align} 
Taking the total expectation of \eqref{eq:mu_k_squared1} and \eqref{eq:mu_k1} we have 
\begin{align*}
    \EE \left[\Lambda_{k}^2\right] \leq \frac{4}{k^{2+ \beta}} \text{ and } \EE \left[\Lambda_{k}\right] \leq \frac{2}{k^{1+ 0.5\beta}}.
\end{align*} 
Now we consider the quantity $\EE \left[\Lambda_{k}\Lambda_{i}| \pastone\right] $ where $i< k $. 
From the definition of  the filtration $\pastone$ and $\Lambda_{i}$,  {we have that} $\Lambda_i \in \pastone$ for all $i < k$. Therefore by the rules of conditional probability
\begin{align*}
    \EE \left[\Lambda_{k}\Lambda_{i}| \pastone\right] =  \EE \left[\Lambda_{k}| \pastone\right] \Lambda_i .
\end{align*}
Hence we can bound the expect value of $\Lambda_{k}\Lambda_{i}$ as follows: 
\begin{align*}
    \EE \left[\Lambda_{k}\Lambda_{i}  \right] 
    &=\EE \left[\EE \left[\Lambda_{k}\Lambda_{i}| \pastone \right] \right]  =\EE \left[\EE \left[\Lambda_{k}| \pastone\right] \Lambda_i  \right] 
    \nonumber\\ \nonumber
    &  {\leq} \EE \left[\EE \left[\frac{2}{k^{1+ 0.5\beta}} \Lambda_i  \right]  \right]   = \frac{2}{k^{1+ 0.5\beta}} \EE \left[\EE \left[ \Lambda_i | \pastonei \right]  \right]
    \nonumber\\ \nonumber
    &  {\leq} \frac{2}{k^{1+ 0.5\beta}} \frac{2}{i^{1+ 0.5\beta}} .
\end{align*}
Since $N_\epsilon$ is a hitting time and the random variables $\Lambda_k^2$, $\Lambda_{k}\Lambda_{i}$ have finite means and are summable (over $\{k = 1,2, \ldots\}$ and over $\{i, k = 1,2,\ldots\}$ respectively), from the general Wald's identity  {\cite{Wald}} we have: 
\begin{align*}
\EE \left[\sum_{k=1}^{N_\epsilon} \Lambda_{k} ^2  \right] &= \EE \left(\sum_{k=1}^{N_\epsilon}\EE \left[  \Lambda_{k} ^2  \right]  \right)\nonumber\\
\EE \left[ \sum_{k=1}^{N_\epsilon}\Lambda_{k} \cdot \left(\sum_{i=1}^{k-1} \Lambda_{i} \right)  \right] &= \EE \left(\sum_{k=1}^{N_\epsilon}\sum_{i=1}^{k-1}\EE \left[  \Lambda_{k} \Lambda_{i}  \right] \right) .
\end{align*}
Applying the previous total expected value bounds in these inequality we have: 
 \begin{align*}
\EE \left[\sum_{k=1}^{N_\epsilon} \Lambda_{k} ^2  \right] & = \EE \left(\sum_{k=1}^{N_\epsilon}\EE \left[  \Lambda_{k} ^2  \right]  \right)
\leq 4\EE \left(\sum_{k=1}^{N_\epsilon}\frac{1}{k^{2+ \beta}}\right) \leq  \frac{4(\beta +2)}{\beta+1}\nonumber\\
 \EE \left[ \sum_{k=1}^{N_\epsilon}\Lambda_{k} \left(\sum_{i=1}^{k-1} \Lambda_{i} \right)  \right] &= \EE \left(\sum_{k=1}^{N_\epsilon}\sum_{i=1}^{k-1}\EE \left[  \Lambda_{k} \Lambda_{i}  \right] \right) \leq 4\EE \left(\sum_{k=1}^{N_\epsilon}\frac{1}{k^{1+ 0.5\beta}} \sum_{i=1}^{k-1} \frac{1}{i^{1+ 0.5\beta}}\right) \leq \frac{4(\beta + 2)}{\beta},
\end{align*}
where the last computations follows from 
\begin{align*}
    &\sum_{k=1}^{N_\epsilon}\frac{1}{k^{2+ \beta}} < 1 + \sum_{k=2}^{\infty} \frac{1}{k^{2+ \beta}} < 1+  \int_{1}^{\infty} \frac{1}{x^{2+ \beta}} dx =  \frac{\beta +2}{\beta+1},\nonumber\\
    &\sum_{k=1}^{N_\epsilon}\frac{1}{k^{1+ 0.5\beta}} < 1 + \sum_{k=2}^{\infty} \frac{1}{k^{1+ 0.5\beta}} < 1+  \int_{1}^{\infty} \frac{1}{x^{1+ 0.5\beta}} dx = \frac{\beta + 2}{\beta}.
\end{align*}
Substituting these bounds into the previous inequality \eqref{eq:bound1}, we have
\begin{align*}
\EE [N_G ] 
 &\leq \frac{1}{2\stepsizebar\epsilon} \left [2\|x_0 - x^* \|^2 + \frac{4(\beta + 2)^2}{\beta^2} + \frac{8(\beta+2)}{\beta+1} 
 \right]. 
\end{align*}
which is the statement of Theorem \ref{th:bound_expectation_N_TS1}.
~\qed \end{proof}

The final complexity result is a straightforward corollary of Theorem \ref{th:framework} and \ref{th:bound_expectation_N_TS1}. 
\begin{theorem}\label{th:final_results_ista}
Under Assumptions \ref{assumpt:convex_smooth}, \ref{assumpt:bounded_below}, \ref{assumpt:key} and Assumption \ref{assumpt:ista_noise} the number of iterations $N_\epsilon$ to reach an $\epsilon$ accuracy solution satisfies the following bound: 
\begin{align*}
\EE (N_\epsilon) \leq \frac{2p}{(2p-1)^2}\left (\frac{1}{\stepsizebar\epsilon} \left [2\|x_0 - x^* \|^2 + \frac{4(\beta + 2)^2}{\beta^2} + \frac{8(\beta+2)}{\beta+1} 
 \right]+ \log_\gamma\left(\frac{\stepsizebar}{\alpha_1}\right)\right)+1,
\end{align*}
where  $\stepsizebar $ is defined in Lemma \ref{lem:good_est_good_model_successful}.
\end{theorem}

  {
\paragraph{\bf Discussion of Condition \eqref{eq:ista_noise}.} 
Condition \eqref{eq:ista_noise} imposes the rate of decay of the variance of the gradient to be faster  than ${\cal O}(\frac{1}{k^2})$, where $k$ is the iteration index. In the case when $G_k$ is computed by sample averaging, this, in principle, implies 
samples size growing at the rate of ${\cal O}({k^2})$. Since our goal here is to derive expected iteration complexity of stochastic ISTA, to reach $\epsilon$ accuracy in terms of function value, to be ${\cal O}(\frac{1}{\epsilon})$, this would imply total sample complexity to be ${\cal O}(\frac{1}{\epsilon^3})$, which may not be optimal. This issue remains even in the smooth case.  Note that Assumption \ref{assumpt:ista_noise}
is in addition to  condition \eqref{eq:first_order1}. In   \cite{Berahas-Cao-Scheinberg}  condition  \eqref{eq:first_order1} (for the smooth case) alone is enough to to derive ${\cal O}(\frac{1}{\epsilon})$ expected complexity  for a step search algorithm applied to convex smooth functions. However, in   \cite{Berahas-Cao-Scheinberg}, in the analysis of the convex case, $\nabla \phi(x_k)$ can only be bounded below  by  
  ${\cal O}({\epsilon})$, thus to satisfy condition \eqref{eq:first_order1} one may need a sample size of ${\cal O}(\frac{1}{\epsilon^2})$, which would result in the same complexity. This complexity is not optimal for convex optimization, when gradient estimates are unbiased and then one seeks to bound expected convergence rate. However, the optimal complexity is unknown, to the best of our knowledge, for the cases of biased gradient estimates and for the bound on expected convergence rate.    
  }
  
  {  Analysis used in 
 \cite{Berahas-Cao-Scheinberg}  does not involve a telescopic sum argument as we do here. On the other hand an assumption that all iterates remain in a bounded region is imposed, while we do not make such an assumption. 
We derive relevant bounds  using \eqref{eq:ista_noise} and the  analysis as extended from \cite{Schmidt-Roux-Bach}. This analysis may not be tight in all cases, as it does not take into account a possible dependence of $\|x_k-x^*\|$ on $\epsilon$ (e.g. $\|x_k-x^*\|$ will decrease as ${\cal O}(\frac{1}{\sqrt{\epsilon}})$ if the function is locally strongly convex. But in general  this dependence can be very weak for general convex cases (for example consider $\phi(x)=x^{2m}$ for any natural number $m$). 
In summary, whether we attain an optimal complexity with our analysis remains unknown. 
The key objective of the paper is deriving analysis for stochastic   FISTA in the next session, where the telescopic sum argument is essential. Optimal complexity is also unknown for this setting.  Considering  strongly convex cases for both ISTA and FISTA in this setting, which may help with some of the issues raised here,  is the subject of future work.
}

  {Finally, there is an additional constant $\beta>0$ in Assumption \ref{assumpt:ista_noise} which also appears in the final expected complexity bound. The choice of this constant and whether or not its effect on the complexity is tight would depend on the specific properties of the first order oracle and its cost.}

\section{Stochastic FISTA Step Search Algorithm}\label{sec:fista}

\subsection{The Algorithm}\label{subsec:outline_algo}
We now extend the analysis of the previous section to an accelerated proximal method. The method that we study here  is a stochastic version
 of the  algorithm FISTA-BKTR in \cite{Scheinberg-Goldfarb-Bai}, which is, in turn, a full backtracking version of the FISTA Algorithm in \cite{Beck-Teboulle}. 
 In FISTA algorithm, the step-size  parameter $\alpha_k$ is chosen by backtracking using the same decrease condition as in ISTA. The key difference is that for the accelerated complexity bound to hold, the step-size  parameter may never increase, it can only decrease. This presents a major obstacle in developing  a stochastic version of this algorithm because, when gradient estimates are stochastic, they may not lead to function decrease even if $\alpha_k$ is arbitrarily small. Therefore a stochastic backtracking algorithm like FISTA will inevitably keep reducing $\alpha_k$ and loose convergence properties. One remedy against this problem is to have sufficient decrease condition based on the stochastic  approximation of $F(x)$ instead of the function itself as is proposed in  \cite{Vaswani-et-al} for smooth gradient method and  in  \cite{Franchini-Porta-Ruggiero-Trombini} for a proximal gradient algorithm (neither are accelerated). In each iteration  $f(x)$ is approximated by a stochastic function whose gradient is $g$ and whose value is used to measure sufficient decrease. The lower bound on $\alpha_k$ is then obtained as $1/L_{max}$ where $L_{max}$ is the largest Lipschitz constant of any gradient of any stochastic approximation function. This not only requires strong assumptions on individual stochastic approximations of $f(x)$, but also results in inferior convergence properties, since the measured  decrease is not the true decrease in the objective function \cite{Jin-Scheinberg-Xie}. 
 
Here we resolve the lower bound issue by choosing to work with a full backtracking version of FISTA that was proposed and analyzed in \cite{Scheinberg-Goldfarb-Bai} and  develop stochastic variant of this method. As in the case of stochastic ISTA analyzed in the previous section, $\alpha_k$ does not have to be bounded from below, but it tends to increase, once it becomes small enough.

We first present that deterministic algorithm and then show how it is modified for the stochastic first-order oracles. Compared to  FISTA  in \cite{Beck-Teboulle}, FISTA-BKTR in \cite{Scheinberg-Goldfarb-Bai} uses an additional  parameter $\theta_k$, which keeps record of the changes in  $\alpha_k$  and 
utilizes the following update  
\begin{align}\label{eq:fista_step}
   (t_{new}, y) = \text{FISTAStep}(x, x^{prev}, t, \theta).
\end{align}
where FISTAStep computes 
\begin{align*}
    t_{new} &=  \frac{1 + \sqrt{1+4\theta t^2}}{2}, \text{ and }
    y = x + \frac{t-1}{t_{new}} ( x - x^{prev} ).
\end{align*}

 \begin{algorithm}[ht]
\caption{FISTA-BKTR}
\label{alg:fista_bktr_orig}
\begin{algorithmic}[1]
\STATE {Initialization:} Choose an initial point $x_0\in \mathbb{R}^d$. Set $t_1 = 1, t_0 = 0$, $y_1=x_0=x_{-1}$  \nonumber\\
Set $0<\gamma<1$, $\theta_0 = 1$ with $\alpha^0_1>0$. 

\FOR{$k=1,2,\cdots,$}
    \STATE Set $\alpha_k=\alpha_k^0$\nonumber\\
     \STATE Compute the gradient $\nabla f(y_k)$ and the reference point $p_{\alpha_k} (y_k) :=\text{prox}_{\alpha_k h} (y_k- \alpha_k \nabla f(y_k) )$
    \STATE Check sufficient decrease condition $F(p_{\alpha_k} (y_k)) \leq Q_{\alpha_k}(p_{\alpha_k} (y_k), y_k)$
        \STATE \hspace{1.25em}If unsuccessful: 
             \STATE \hspace{2.5em}  Compute $(t_k, y_k) = \text{FISTAStep}(x_{k-1}, x_{k-2}, t_{k-1}, \theta_{k-1}) $. Set $\alpha_{k+1} = \gamma\alpha_k$ and $\theta_{k} = \tfrac{\theta_{k-1}}{\gamma}$.
            \STATE \hspace{2.5em}  Return to Step 5. 
        \STATE \hspace{1.25em}If successful: 
            \STATE \hspace{2.5em}Set $x_k=p_{\alpha_k} (y_k)$ and $t_k = t_k^{next}$,
    then choose $\alpha^0_{k+1}>0$ and $\theta_{k} = \tfrac{\alpha_k}{\alpha^0_{k+1}}$. \\
  \STATE \hspace{2.5em}$(t_{k+1}, y_{k+1}) = \text{FISTAStep}(x_{k},  {x_{k-1}}, t_{k}, \theta_{k}) $
\ENDFOR
\end{algorithmic}
\end{algorithm}

Algorithm \ref{alg:fista_bktr_orig} is analyzed in  \cite{Scheinberg-Goldfarb-Bai} using the key property in the analysis of FISTA that $ \alpha_{k-1} t_{k-1}^2\geq \alpha_k t_k(t_{k}-1)$ has to be maintained at each iteration, while also enabling $t_k$ to grow as a linear function of $k$.
In the case of the original FISTA algorithm this is simply enforced by maintaining
$ \alpha_{k-1} \geq \alpha_k $ and $t_{k-1}^2= t_k(t_{k}-1)$, while  FISTA-BKTR uses parameter $\theta_k$ to offset increases and decreases in
 $\alpha_{k}$ and maintain the key properties. 
 
 Here we will use the main FISTA-BKTR  scheme, however, as in the case of stochastic ISTA, the gradient estimate needs to be recomputed each time a step is attempted, even when the iteration is unsuccessful. Therefore the iteration counter gets increased on unsuccessful steps as well as successful ones and the algorithm notation needs to be updated accordingly. In particular 
we use  $x^{prev}_k$ (instead of $x_{k-1}$) to denote the iterate of the last successful iteration, that is the last iterate different from $x_k$. 
Similarly, $t_k^{next}$ denotes the value of $t_{k+1}$ that is assigned if iteration $k$ is successful, otherwise $t_{k+1}=t_k$. 
We also need to keep track of the $\alpha_k^{succ}$ - the step-size  parameter of the last successful iteration up to iteration $k$. 
On the other hand, $\alpha_k^0$ value is not used since the dynamics of the step-size  parameter are simplified - it is decreased by the factor of $\gamma$ on each unsuccessful iteration and increased proportionally on each successful one. 

The complexity analysis for the stochastic case necessarily has to count all successful and unsuccessful iterations, 
since the first-order oracle is called at each such iteration. For the deterministic FISTA-BKTR complexity analysis only includes the bound on successful iterations,
firstly because the first-order oracle is only needed when an iterate gets updated and because each successful iteration performs an inner loop whose length is deterministically bounded by ${\cal O}(\log_\gamma(\frac{1}{L}))$. 


\begin{algorithm}[ht]
\caption{Stochastic FISTA Step Search Algorithm}
\label{alg:fista_bktr_line_search}
\begin{algorithmic}[1]
\STATE {Initialization:} Choose an initial point $x_0\in \mathbb{R}^d$. Set $t_0 = 0$, let $x^{prev}_0 = x_0$. \nonumber\\
Set $0<\gamma<1$, $\theta_0 = \gamma$ with $\alpha_1>0$.  Let oracle SFO$_{\constG, p}$ be given with  {$0 \leq \constG \leq \frac{1}{3}$} and $\frac{1}{2} < p \leq 1 $.
\FOR{$k=1,2,\cdots,$}
    \STATE Compute $(t_k^{next}, y_k) = \text{FISTAStep}(x_{k-1}, x^{prev}_{k-1}, t_{k-1}, \theta_{k-1}) $
    \STATE Compute the gradient estimate $g_k=G(y_k,\xi_k)$ using the stochastic first-order oracle SFO$_{\constG, p}(y_k,\alpha_k)$.
    Compute the reference point $\tilde{p}_{\alpha_k} (y_k) =\text{prox}_{\alpha_k h} (y_k- \alpha_k g_k)$.
    \STATE Check sufficient decrease condition $F(\tilde{p}_{\alpha_k} (y_k)) \leq \tilde{Q}_{\alpha_k}(\tilde{p}_{\alpha_k} (y_k), y_k)$
        \STATE \hspace{1.25em}If unsuccessful: 
            \STATE \hspace{2.5em}Set $\left(x_{k},x^{prev}_{k},t_k\right) = \left(x_{k-1}, x^{prev}_{k-1}, t_{k-1}\right)$, then set $\alpha_{k+1} = \gamma\alpha_k$, 
            and $\theta_{k} = \tfrac{\theta_{k-1}}{\gamma}$.
        \STATE \hspace{1.25em}If successful: 
            \STATE \hspace{2.5em}Set $\left(x_k, x^{prev}_{k},t_k\right) = \left(\tilde{p}_{\alpha_k} (y_k), x_{k-1}, t_k^{next}\right)$,
    then set $\alpha_{k+1} = \tfrac{\alpha_k}{\gamma},$ 
    and $\theta_{k} = \gamma$.
\ENDFOR
\end{algorithmic}
\end{algorithm}


Similar to Algorithm \ref{alg:ista_bktr_line_search}, 
Algorithm \ref{alg:fista_bktr_line_search} generates a stochastic process $\{Y_k, X_k, G_k, \Stepsize{k}, T_k\}$ with realizations $\{y_k, x_k, g_k, \alpha_k, t_k\}$. 
We note that the steps 4 and 5 of Algorithm \ref{alg:fista_bktr_line_search} are exactly as Algorithm \ref{alg:ista_bktr_line_search}.
Since the dynamics of the step-size parameter is based on the reference points $y_k$, $\tilde{p}_{\alpha_k} (y_k)$ in Step 4 and the sufficient decrease condition in Step 5, 
the results in Sections \ref{sec:stoc_proc} and \ref{subsec:successful}  hold  for Algorithm \ref{alg:fista_bktr_line_search} under the same assumptions as for 
Algorithm \ref{alg:ista_bktr_line_search}.

In addition to these results,  we prove the following lemma for the dynamics of $ (\alpha_k, t_k)$ in Algorithm  \ref{alg:fista_bktr_line_search}, using similar principles as in \cite{Scheinberg-Goldfarb-Bai}.
\begin{lemma}\label{lem:fista_linesearch}
 We denote $\alpha^{succ}_k$ be the step-size parameter  at the successful iteration up to iteration $k$: 
\begin{align*}
    \alpha^{succ}_k = \alpha_{k'}, \text{ where } k' = \max \{k'' \text{ successful }, k'' \leq k\} .
\end{align*}
We also adopt the convention that $\alpha^{succ}_k = \alpha_0 = \gamma \alpha_1$ if there is no successful iteration up to iteration $k$. Then we have the following properties:

1. $\alpha^{succ}_{k-1} t_{k-1}^2\geq \alpha_k t_k(t_{k}-1)$ for every successful iteration $k$.

2. $\alpha^{succ}_{K} t_{K}^2 \geq \left(\sum_{k \in S, k \le K } \sqrt{\alpha_k} /2\right)^2$ for any number of iterations $K$.

\end{lemma}

\begin{proof}
1. From the update of the algorithm and the fact that $k$ is successful: 
\begin{align*}
    t_k = t_k^{next} &=  \frac{1 + \sqrt{1+4\theta_{k-1} t_{k-1}^2}}{2},
\end{align*}
we have $ \theta_{k-1} t_{k-1}^2=  t_k(t_{k}-1)$.
The statement of the lemma $\alpha^{succ}_{k-1} t_{k-1}^2\geq \alpha_k t_k(t_{k}-1)$ is equivalent to  
\begin{align*}
\alpha^{succ}_{k-1} \geq \alpha_k \theta_{k-1}.
\end{align*}
Let $k'$ be the previous successful iteration prior to $k$ ($k' < k$), i.e. $\alpha^{succ}_{k-1} = \alpha_{k'}$ and $\alpha^{succ}_{k-1} = \alpha_{0}$ if there is no successful iteration prior to $k$. By the definition of $k'$ we have 
$\alpha_{k'+1} = \tfrac{\alpha_{k'}}{\gamma},$ and $\theta_{k'} = \gamma $.
Since the iterations from $k' + 1$ to $k-1$ are unsuccessful (if $k' <k-1$): 
\begin{align*}
    \alpha_{k} &= \alpha_{k'+1} \cdot  \gamma^{k-k'-1}\nonumber\\
    \theta_{k-1} &= \theta_{k'} \gamma^{k'+1 -k}.
\end{align*}
Therefore 
\begin{align*}
    \alpha_{k} &= \alpha_{k'+1} \cdot  \gamma^{k-k'-1}= \alpha_{k'} \cdot  \gamma^{k-k'-2} \nonumber\\
    \theta_{k-1} &= \theta_{k'} \gamma^{k'+1 -k} = \gamma^{k'+2 -k},
\end{align*}
and it proves the statement: 
    \begin{align}\label{eq:theta}
    \alpha_k \theta_{k-1} = \alpha_{k'} \cdot  \gamma^{k-k'-2} \cdot \gamma^{k'+2 -k} = \alpha_{k'}. 
    \end{align}



2. First, we prove the statement for all iterations $K$ where $K$ is successful. We prove this by induction. 
\begin{itemize}
\item Let $k_1 \geq 1$ be the first successful iteration of the algorithm. By definition we have 
$\alpha^{succ}_{k_1} = \alpha_{k_1}, \text{ since } k_1 = \max \{k'' \text{ successful }, k'' \leq k_1\} $.  
Since $k_1$ is successful:
\begin{align*}
    t_{k_1} = t_{k_1}^{next} = \frac{1 + \sqrt{1+4\theta_{(k_1-1)} t_{(k_1-1)}^2}}{2}= \frac{1 + \sqrt{1+4\theta_{(k_1-1)} \cdot 0}}{2} = 1.
\end{align*}
Therefore the statement holds for $k_1$: 
\begin{align*}
    \alpha^{succ}_{k_1} t_{k_1}^2 =  \alpha_{k_1} \cdot 1 \geq \left( \sqrt{\alpha_{k_1}} /2\right)^2 = \left(\sum_{k \in S, k \le k_1 } \sqrt{\alpha_k} /2\right)^2.
\end{align*}
\item Now for an iteration $K$ where $K$ is successful, we assume the statement holds for the successful iteration $k'$ prior to $K$ ($k'$ is the largest index of a successful iteration satisfying $k' < K$). We have 
$ \alpha^{succ}_{k'} t_{k'}^2 \geq \left(\sum_{k \in S, k \le k' } \sqrt{\alpha_k} /2\right)^2$ and  
\begin{align*}
    \sqrt{\alpha^{succ}_{k'}} t_{k'} \geq \sum_{k \in S, k \le k' } \sqrt{\alpha_k} /2.
\end{align*}

Since both $K$ and $k'$ are successful, $ \alpha^{succ}_{K} =  \alpha_{K} $ and $ \alpha^{succ}_{k'} =  \alpha_{k'} $.
Using the fact that $K$ is successful: 
\begin{align*}
    t_{K} &= t_{K}^{next} = \frac{1 + \sqrt{1+4\theta_{(K-1)} t_{(K-1)}^2}}{2} \geq \frac{1}{2} +  \sqrt{\theta_{(K-1)}}  t_{(K-1)}\\
    &\geq \frac{1}{2} +  \sqrt{\theta_{(K-1)} } \cdot t_{k'} \geq \frac{1}{2} +  \sqrt{\frac{\alpha_{k'}}{\alpha_{K}} } \cdot t_{k'}.
\end{align*}
where the last inequality follows from previous argument \eqref{eq:theta}.
Hence we have 
\begin{align*}
    \sqrt{\alpha^{succ}_{K}} t_{K} = \sqrt{\alpha_{K}} t_{K} \geq \sqrt{\alpha_{k'}} t_{k'} + \sqrt{\alpha_K} /2 = \sqrt{\alpha^{succ}_{k'}} t_{k'} + \sqrt{\alpha_K} /2.
\end{align*}
Therefore 
\begin{align*}
    \sqrt{\alpha^{succ}_{K}} t_{K} \geq \sum_{k \in S, k \le k' } \sqrt{\alpha_k} /2 + \sqrt{\alpha_K} /2 = \sum_{k \in S, k \le K } \sqrt{\alpha_k} /2,
\end{align*}
and the statement holds for iteration $K$:
$ \alpha^{succ}_{K} t_{K}^2 \geq \left(\sum_{k \in S, k \le K } \sqrt{\alpha_k} /2\right)^2$.
\item By induction, the statement holds for all the successful iterations. For an arbitrary unsuccessful iteration $K$, the statement holds for the successful iteration $k'$ prior to $K$ and neither right hand side  {nor} left hand side term changes at  unsuccessful iteration. The statement is also true in the case when there is no successful iteration prior to $K$ as both sides equal to 0 ($t_0 = 0$). Therefore it holds for every iteration. 
\end{itemize}
~\qed \end{proof}

\subsection{Expected Complexity Analysis of Algorithm \ref{alg:fista_bktr_line_search} }\label{subsec:analysis2}
As in the case of Algorithm \ref{alg:ista_bktr_line_search}, the expected complexity $\EE (N_\epsilon)$ is bounded based on $\EE(N_G)$ (Theorem \ref{th:framework}). Therefore all we need is to bound the expected number of {\em good} iterations of Algorithm \ref{alg:fista_bktr_line_search}. We consider the following quantity: 
\begin{align}\label{eq:define_m}
    m_k = 2\alpha^{succ}_{k} t_{k}^2v_{k} +\norm{u_k }^2,
\end{align}
where $v_k =  F(x_k) - F^* $ and $u_k = (t_{k}-1)x^{prev}_{k} + x^* - t_k x_k $. 
While $v_k$ denotes the usual optimality gap, $u_k$ resembles the usual term in deterministic FISTA algorithm analysis \cite{Beck-Teboulle,Scheinberg-Goldfarb-Bai}.
The next Lemma states that  $m_{k}$ do not change during unsuccessful iterations. 
\begin{lemma}[$m_{k}$ on unsuccsessful steps]\label{lem:unsuccessful_step_bound}
Let $m_k$ be defined as in \eqref{eq:define_m}. 
If iteration $k$ of Algorithm \ref{alg:fista_bktr_line_search}  is unsuccessful, then we have 
\begin{align*}m_{k} = m_{k-1}. \end{align*}
\end{lemma}

Lemma \ref{lem:successful_step_bound} 
shows how the quantity $m_k$ changes in the successful steps of Algorithm \ref{alg:fista_bktr_line_search}, under the assumptions of convexity and $L$-Lipschitz smoothness of the objective function. 
This Lemma is a stochastic version of the deterministic bound in Lemma 3.3 \cite{Scheinberg-Goldfarb-Bai} for FISTA-BKTR Algorithm. 

\begin{lemma}[$m_k$ on  successful step]\label{lem:successful_step_bound}
Under Assumption \ref{assumpt:convex_smooth}, if $k$ is a successful iteration of Algorithm \ref{alg:fista_bktr_line_search}, then we have:
\begin{align*}
&2\alpha^{succ}_{k}t_k^2 v_k -2\alpha^{succ}_{k-1} t_{k-1}^2v_{k-1} \leq \norm{u_{k-1}}^2 -  \norm{u_k }^2 + \lambda_k \cdot \| u_{k-1}\|,
\end{align*}
where 
\begin{align*}
\lambda_k =  \|2\alpha_kt_k (\nabla f(y_k )- g_k )\|.
\end{align*}
Equivalently, 
\begin{align*}
m_{k} -m_{k-1} \leq \lambda_k \cdot \|u_{k-1}\|.
\end{align*}
\end{lemma}

\begin{proof}
Since $k$ is a successful iteration, we have $F(\tilde{p}_{\alpha_k} (y_k)) \leq \tilde{Q}_{\alpha_k}(\tilde{p}_{\alpha_k} (y_k), y_k)$. 
Also $x_k = \tilde{p}_{\alpha_k} (y_k) $. 
Hence 
\begin{align*}
    F(x_k) = f(x_k) + h(x_k) \leq f(y_k) +(x_k -y_k)^\top g_k + \frac{1}{2\alpha_k}\|x_k-y_k\|^2 + h(x_k) .
\end{align*}
By the  convexity of $f$ at $x$ and $y_k$ and the convexity of $h$ at $x$ and $x_k$:
\begin{align*}
f(x) \geq f(y_k) + \nabla f(y_k )^\top (x-y_k),
\nonumber\\
h(x) \geq h(x_k) + \gamma(y_k)^\top (x-x_k),
\end{align*}
where the second bound holds for every subgradient $\gamma(y_k)$ of $h$ at $x_k$. 
We have 
\begin{align*}
F(x) &= f(x) + h(x)\geq f(y_k)+ h(x_k) + \nabla f(y_k )^\top (x-y_k)  + \gamma(y_k)^\top (x-x_k)
\nonumber\\
&= f(x) + h(x)\geq f(y_k)+ h(x_k) + [\nabla f(y_k )+ \gamma(y_k)]^\top (x-y_k)  + \gamma(y_k)^\top ( y_k -x_k) 
\end{align*}
Combine this with the previous bound of $F(x_k)$ we have 
\begin{align*}
&F(x) - F(x_k) \nonumber\\&\geq  [\nabla f(y_k )+ \gamma(y_k)]^\top (x-y_k)  + \gamma(y_k)^\top ( y_k -x_k) - (x_k -y_k)^\top g_k - \frac{1}{2\alpha_k}\|x_k-y_k\|^2\nonumber\\
&=  [\nabla f(y_k )+ \gamma(y_k)]^\top (x-y_k)  + (\gamma(y_k)+g_k)^\top ( y_k -x_k)  - \frac{1}{2\alpha_k}\|x_k-y_k\|^2.
\end{align*}
Using the fact that $x_k = \tilde{p}_{\alpha_k} (y_k) =\arg\min_y \tilde{Q}_{\alpha_k}(y, y_k)$, we get that $ y_k - x_k= \alpha_k[\gamma(y_k)+g_k]$ for some subgradient $\gamma(y_k)$ of $h$ at $x_k$. Therefore
\begin{align*}
F(x) - F(x_k)
&\geq  [\nabla f(y_k )+ \gamma(y_k)]^\top (x-y_k)  + \alpha_k\|\gamma(y_k)+g_k) \|^2 - \frac{1}{2} \alpha_k \|\gamma(y_k)+g_k)\|^2\nonumber\\
&=  [\nabla f(y_k )+ \gamma(y_k)]^\top (x-y_k)  + \frac{\alpha_k}{2}\|\gamma(y_k)+g_k) \|^2 .
\end{align*}
Apply this equation for $x = x_{k-1}$ and $x=x^*$ we have 
\begin{align*}
F(x_{k-1}) - F(x_k) &\geq  [\nabla f(y_k )+ \gamma(y_k)]^\top (x_{k-1}-y_k)  + \frac{\alpha_k}{2}\|\gamma(y_k)+g_k) \|^2 ,
\nonumber\\
F(x^*) - F(x_k) &\geq  [\nabla f(y_k )+ \gamma(y_k)]^\top (x^*-y_k)   + \frac{\alpha_k}{2}\|\gamma(y_k)+g_k) \|^2 .
\end{align*}
Let $v_k =  F(x_k) - F(x^*) $ then 
\begin{align*}
v_{k-1} - v_k &\geq  [\nabla f(y_k )+ \gamma(y_k)]^\top (x_{k-1}-y_k)  + \frac{\alpha_k}{2}\|\gamma(y_k)+g_k) \|^2  ,
\nonumber\\
-v_k &\geq  [\nabla f(y_k )+ \gamma(y_k)]^\top (x^*-y_k)   + \frac{\alpha_k}{2}\|\gamma(y_k)+g_k) \|^2  .
\end{align*}
Multiplying the first equality by $t_{k}-1$ and adding to the second equality, we get
\begin{align*}
(t_{k}-1)v_{k-1} - t_k v_k &\geq  [\nabla f(y_k )+ \gamma(y_k)]^\top ((t_{k}-1)x_{k-1} + x^* - t_k y_k )   + \frac{\alpha_k}{2} t_k \|\gamma(y_k)+g_k) \|^2 .
\end{align*}
Multiplying the previous inequality by $t_k$ we have
 \begin{align*}
 &t_k(t_{k}-1)v_{k-1} - t_k^2 v_k \nonumber\\
 &\geq t_k [\nabla f(y_k )+ \gamma(y_k)]^\top ((t_{k}-1)x_{k-1} + x^* - t_k y_k )   + \frac{\alpha_k}{2}t_k^2\|\gamma(y_k)+g_k) \|^2 \nonumber\\
&=  t_k [\nabla f(y_k )- g_k +g_k + \gamma(y_k)]^\top ((t_{k}-1)x_{k-1} + x^* - t_k y_k )  + \frac{\alpha_k}{2}t_k^2\|\gamma(y_k)+g_k) \|^2
\nonumber\\
&=  t_k [g_k + \gamma(y_k)]^\top ((t_{k}-1)x_{k-1} + x^* - t_k y_k )  + \frac{\alpha_k}{2}t_k^2\|\gamma(y_k)+g_k) \|^2
 \nonumber\\&
+ t_k [\nabla f(y_k )- g_k ]^\top ((t_{k}-1)x_{k-1} + x^* - t_k y_k ) 
\nonumber\\
&=  t_k \frac{1}{\alpha_k}(y_k - x_k)^\top ((t_{k}-1)x_{k-1} + x^* - t_k y_k )  + \frac{\alpha_k}{2}t_k^2\|\gamma(y_k)+g_k) \|^2
 \nonumber\\&
+ t_k [\nabla f(y_k )- g_k ]^\top ((t_{k}-1)x_{k-1} + x^* - t_k y_k ) 
\nonumber\\
&=  t_k 
\frac{1}{\alpha_k}(y_k - x_k)^\top ((t_{k}-1)x_{k-1} + x^* - t_k x_k )   + t_k 
\frac{1}{\alpha_k}(y_k - x_k)^\top t_k(x_k -y_k)  \nonumber\\&
+ \frac{1}{2\alpha_k}t_k^2\|x_k -y_k \|^2
+ t_k [\nabla f(y_k )- g_k ]^\top ((t_{k}-1)x_{k-1} + x^* - t_k y_k ) 
\nonumber\\
&=  t_k 
\frac{1}{\alpha_k}(y_k - x_k)^\top ((t_{k}-1)x_{k-1} + x^* - t_k x_k )   -  
\frac{1}{2\alpha_k}\norm{t_k (x_k -y_k)}^2 \nonumber\\&
+ t_k [\nabla f(y_k )- g_k ]^\top ((t_{k}-1)x_{k-1} + x^* - t_k y_k ) .
\end{align*}
Multiplying the final statement by $2\alpha_k $ we get 
 \begin{align}\label{eq_bound_mu_t}
&~2\alpha_k t_k(t_{k}-1)v_{k-1} - 2\alpha_kt_k^2 v_k \nonumber\\
&\geq 2t_k 
(y_k - x_k)^\top ((t_{k}-1)x_{k-1} + x^* - t_k x_k )   -  
\norm{t_k (x_k -y_k)}^2  \nonumber\\
&  + 2\alpha_kt_k [\nabla f(y_k )- g_k ]^\top ((t_{k}-1)x_{k-1} + x^* - t_k y_k ) .
\end{align}
Now we consider the following term: $2t_k 
(y_k - x_k)^\top ((t_{k}-1)x_{k-1} + x^* - t_k x_k )   -  
\norm{t_k (x_k -y_k)}^2$.
Let $a  = t_k 
(y_k - x_k), b = (t_{k}-1)x_{k-1} + x^* - t_k x_k$ we have 
\begin{align}\label{eq:03}
    & \quad \ 2t_k 
(y_k - x_k)^\top ((t_{k}-1)x_{k-1} + x^* - t_k x_k )   -  
\norm{t_k (x_k -y_k)}^2\nonumber\\ &=  2ab - a^2 \nonumber\\
&= b^2 -b^2 +2ab - a^2= b^2 -(a-b)^2\nonumber\\
&=\norm{(t_{k}-1)x_{k-1} + x^* - t_k x_k }^2 -\norm{(t_{k}-1)x_{k-1} + x^* - t_k x_k  -t_k (y_k - x_k)}^2\nonumber\\
&=\norm{(t_{k}-1)x_{k-1} + x^* - t_k x_k }^2 -\norm{(t_{k}-1)x_{k-1} + x^* - t_k y_k  }^2.
\end{align}
Since $k$ is a successful iteration, we have that $x^{prev}_{k} = x_{k-1}$. From the definition of $u_k$
\begin{align*}u_k = (t_{k}-1)x^{prev}_{k} + x^* - t_k x_k, \end{align*} 
we have that 
\begin{align*}u_k = (t_{k}-1)x_{k-1} + x^* - t_k x_k. \end{align*} 
Now from the definition of FISTAStep \eqref{eq:fista_step}
\begin{align*}y_k = x_{k-1} + \frac{t_{k-1}-1}{t_{k}} ( x_{k-1} - x^{prev}_{k-1} ),\end{align*}
we have 
\begin{align*}t_k y_k = t_kx_{k-1} + (t_{k-1}-1) ( x_{k-1} - x^{prev}_{k-1} ).\end{align*}
Combining these equations with \eqref{eq:03} we get: 
\begin{align*}
& \quad \ 2t_k 
(y_k - x_k)^\top ((t_{k}-1)x_{k-1} + x^* - t_k x_k )   -  
\norm{t_k (x_k -y_k)}^2\nonumber\\
&=\norm{(t_{k}-1)x_{k-1} + x^* - t_k x_k }^2 -\norm{(t_{k}-1)x_{k-1} + x^* - t_k y_k  }^2\nonumber\\
&=\norm{u_k }^2 -\norm{(t_{k}-1)x_{k-1} + x^* - t_kx_{k-1} - (t_{k-1}-1) ( x_{k-1} -  x^{prev}_{k-1} )  }^2\nonumber\\
&=\norm{u_k }^2 -\norm{-t_{k-1} x_{k-1} + x^*  + (t_{k-1}-1)  x^{prev}_{k-1}  }^2\nonumber\\
&=\norm{u_k }^2 -\norm{-u_{k-1}}^2.
\end{align*}
Substituting back to previous equality \eqref{eq_bound_mu_t} we get 
 \begin{align*}
&2\alpha_k t_k(t_{k}-1)v_{k-1} - 2\alpha_kt_k^2 v_k \nonumber\\
&\geq \norm{u_k }^2 -\norm{u_{k-1}}^2  + 2\alpha_kt_k [\nabla f(y_k )- g_k ]^\top ((t_{k}-1)x_{k-1} + x^* - t_k y_k ) .
\end{align*}
Therefore  
  \begin{align*}
&2\alpha_kt_k^2 v_k -2\alpha_k t_k(t_{k}-1)v_{k-1}\nonumber\\ 
&\leq \norm{u_{k-1}}^2 -  \norm{u_k }^2 + \|2\alpha_kt_k (\nabla f(y_k )- g_k )\| \cdot \| u_{k-1}\|.
\end{align*}
From Lemma \ref{lem:fista_linesearch} we have $\alpha^{succ}_{k-1} t_{k-1}^2\geq \alpha_k t_k(t_{k}-1)$
for every successful iteration $k$. Also since $k$ is successful, we have  $\alpha_k = \alpha^{succ}_{k}$.
Therefore we have the final bound of Lemma \ref{lem:successful_step_bound}.
~\qed \end{proof}

Lemma \ref{lem:successful_step_bound} shows how the term $2\alpha^{succ}_{k} t_{k}^2v_{k} +\norm{u_k }^2$ changes between the successful iterations, where 
$\|2\alpha_kt_k (\nabla f(y_k )- g_k )\| \cdot \| u_{k-1}\|$ dictates how the noise of the model at iteration $k$ affects the bound.

We note that here $\lambda_k$ denotes a slightly different quantity  than in stochastic ISTA step search Algorithm \ref{alg:ista_bktr_line_search}, however, the role of these quantities in the two algorithms is the same, hence we use the same notation. 
In the next lemma, we bound the total accumulation of this error in the bound on the overall accuracy. 
\begin{lemma}\label{lem:after_K}
Given any integer $K$, after $K$ iterations of the algorithm, the following bound holds: 
\begin{align*}
    2\alpha^{succ}_{K} t_{K}^2v_{K} \leq 3\alpha^{succ}_{0} t_{0}^2v_{0} +2\norm{u_0 }^2 + \sum_{k=1}^{K} \lambda_{k} \cdot \left[\sum_{i=1}^{k-1} \lambda_{i} \right]  + 2\sum_{k=1}^{K} \lambda_{k} ^2  .
\end{align*}
\end{lemma}

\begin{proof}
For every $K$, from Lemma \ref{lem:unsuccessful_step_bound}  we have
\begin{align*}
    m_K - m_0 = \sum_{k \in S, k\leq K} (m_{k} -m_{k-1}),
\end{align*}
where $S$ is the set of successful iterations.
We apply Lemma \ref{lem:successful_step_bound} for successful iterations and then 
\begin{align*}
    m_K - m_0 &= \sum_{k \in S, k\leq K} (m_{k} -m_{k-1})\leq \sum_{k \in S, k\leq K} \lambda_k \cdot \|u_{k-1}\| \leq \sum_{k=1}^{K} \lambda_k \cdot \| u_{k-1}  \|  .
\end{align*}
Plugging this and the definition of $m_k$ back to the previous inequality, we have our key bound: 
\begin{align}\label{eq:thm_key}
    2\alpha^{succ}_{K} t_{K}^2v_{K} +\norm{u_K }^2 
    &\leq \left(2\alpha^{succ}_{0} t_{0}^2v_{0} +\norm{u_0 }^2\right) + \sum_{k=1}^{K} \lambda_k \cdot \| u_{k-1}\| .
\end{align}
Similar to the proof of ISTA, 
we follow the argument in \cite{Schmidt-Roux-Bach} with two steps. First, we use the key bound \eqref{eq:thm_key} to bound $\| u_{k-1}\|$ in terms of the noise terms $\lambda_k$. Then we use this bound on $\| u_{k-1}\|$ to substitute into the last sum of the key bound \eqref{eq:thm_key} and obtain the final derivation in terms of the noise $\lambda_k$.
Ignoring the positive term $2\alpha^{succ}_{K} t_{K}^2v_{K}$, we have 
\begin{align*}
    \norm{u_K }^2 
    &\leq \left(2\alpha^{succ}_{0} t_{0}^2v_{0} +\norm{u_0 }^2\right) + \sum_{k=1}^K \lambda_k \cdot \| u_{k-1}\| \nonumber\\
    &= 2\alpha^{succ}_{0} t_{0}^2v_{0} +\norm{u_0 }^2 + \lambda_1 \cdot \| u_{0}\| + \sum_{k=1}^{K} \lambda_{k+1} \cdot \| u_{k}\|.
\end{align*}
This is equivalent to 
\begin{align*}
    \norm{u_K }^2 
    \leq Z + \sum_{k=1}^{K} \lambda_{k+1} \| u_{k}\|,
\end{align*}
where $Z = 2\alpha^{succ}_{0} t_{0}^2v_{0} +\norm{u_0 }^2 + \lambda_1 \| u_{0}\|$.
Using this inequality and applying Lemma 1 in \cite{Schmidt-Roux-Bach} for the non-negative sequence $\{\|u_K\|\}$, we have
\allowdisplaybreaks
\begin{align*}
    \| u_{k}\| &\leq \frac{1}{2} \sum_{i=1}^{k} \lambda_{i+1} + \left( Z + \left(\frac{1}{2} \sum_{i=1}^{k} \lambda_{i+1} \right)^2\right)^{1/2}
    \leq  \sum_{i=1}^{k} \lambda_{i+1} +  Z^{1/2} \nonumber\\
    &\leq  \sum_{i=1}^{k+1} \lambda_{i} + \left(2\alpha^{succ}_{0} t_{0}^2v_{0} +2\norm{u_0 }^2  \right)^{1/2},
\end{align*}
where the last line follows from the fact that
\begin{align*}
    Z^{1/2} &= \sqrt{2\alpha^{succ}_{0} t_{0}^2v_{0} +\norm{u_0 }^2 + \lambda_1 \| u_{0}\|  } \leq  \sqrt{2\alpha^{succ}_{0} t_{0}^2v_{0} +\norm{u_0 }^2 + \lambda_1^2 +\| u_{0}\|^2  } \nonumber\\
    &\leq \sqrt{2\alpha^{succ}_{0} t_{0}^2v_{0} +2\norm{u_0 }^2 + \lambda_1^2  } \leq  \lambda_1 + \left(2\alpha^{succ}_{0} t_{0}^2v_{0} +2\norm{u_0 }^2  \right)^{1/2}.
\end{align*}
Dropping the term $\norm{u_K }^2 $ in 
the inequality \eqref{eq:thm_key} and using the bound on $\| u_{k}\|$ we have: 
\allowdisplaybreaks
\begin{align*}
    &2\alpha^{succ}_{K} t_{K}^2v_{K} \nonumber\\
    &\leq 2\alpha^{succ}_{0} t_{0}^2v_{0} +\norm{u_0 }^2 + \sum_{k=1}^{K} \lambda_{k} \cdot \| u_{k-1}\| \nonumber\\
    &\leq 2\alpha^{succ}_{0} t_{0}^2v_{0} +\norm{u_0 }^2 + \sum_{k=1}^{K} \lambda_{k} \cdot \left[\sum_{i=1}^{k} \lambda_{i} + \left(2\alpha^{succ}_{0} t_{0}^2v_{0} +2\norm{u_0 }^2  \right)^{1/2} \right] \nonumber\\
    &= 2\alpha^{succ}_{0} t_{0}^2v_{0} +\norm{u_0 }^2 + \sum_{k=1}^{K} \lambda_{k} \cdot \left[\sum_{i=1}^{k} \lambda_{i} \right] + \left(2\alpha^{succ}_{0} t_{0}^2v_{0} +2\norm{u_0 }^2  \right)^{1/2} \sum_{k=1}^{K} \lambda_{k}   \nonumber\\
    &\leq 2\alpha^{succ}_{0} t_{0}^2v_{0} +\norm{u_0 }^2 + \sum_{k=1}^{K} \lambda_{k} \cdot \left[\sum_{i=1}^{k-1} \lambda_{i} \right]  + \sum_{k=1}^{K} \lambda_{k} ^2+  \frac{1}{2} \left(2\alpha^{succ}_{0} t_{0}^2v_{0} +2\norm{u_0 }^2  \right)  + \sum_{k=1}^{K} \lambda_{k} ^2\nonumber\\
    &= 3\alpha^{succ}_{0} t_{0}^2v_{0} +2\norm{u_0 }^2 + \sum_{k=1}^{K} \lambda_{k} \cdot \left[\sum_{i=1}^{k-1} \lambda_{i} \right]  + 2\sum_{k=1}^{K} \lambda_{k} ^2,  
\end{align*}
which is our desired bound. 
~\qed 
\end{proof}

We make the following additional assumption on the output of oracle SFO. 
\begin{assumption}\label{assumpt:fista_noise}
At every iteration $k$,  $g_k=g(y_k,\xi_k)$ computed using the stochastic first-order oracle SFO$_{\constG, p}(y_k,\alpha_k)$ satisfies 
\begin{align}\label{eq:fista_noise}
&\EE_{\xi_k} \left[ \| (\nabla f(y_{k} )- G_{k} )\|^2 | \pastone\right] \leq \frac{1}{\alpha_{k}^2t_{k}^2\cdot k^{2+ \beta}},
\end{align}
for some fixed $\beta>0$.
\end{assumption}

This assumption requires the error of the gradient estimate to decay a little bit faster than $\frac{1}{t_k k}$. In  \cite{Schmidt-Roux-Bach}
the deterministic error decays faster than $\frac{1}{k^2}$ to retain accelerated convergence rate. Our requirements here are similar, since due to Lemma \ref{lem:fista_linesearch} and the expected number of {\em good} iterations, we have that $t_k ={\cal O}(k)$. 

Using Lemma \ref{lem:successful_step_bound} and the Assumption \ref{assumpt:fista_noise} we prove the following theorem: 

\begin{theorem}[Upper bound of the number of large true successful steps]\label{th:bound_expectation_N_TS}
Under Assumptions \ref{assumpt:convex_smooth}, \ref{assumpt:bounded_below}
and  \ref{assumpt:fista_noise}, we have  
 \begin{align*}
 \EE [N_G] \leq \sqrt{
 \frac{2}{\stepsizebar\epsilon} \left [3\alpha^{succ}_{0} t_{0}^2v_{0} +2\norm{u_0 }^2 + \frac{4(\beta + 2)^2}{\beta^2} + \frac{8(\beta+2)}{\beta+1}  
 \right] }. 
\end{align*}

\end{theorem}

\begin{proof} 
Since the statement of Lemma \ref{lem:after_K} holds for every realizations of Algorithm \ref{alg:fista_bktr_line_search}, for $K = N_\epsilon -1$ we have: 
\begin{align*}
    2\Alpha^{succ}_{N_\epsilon -1} T_{N_\epsilon -1}^2 (F(X_{N_\epsilon -1}) - F^*)\leq 3\alpha^{succ}_{0} t_{0}^2v_{0} +2\norm{u_0 }^2 + \sum_{k=1}^{N_\epsilon -1} \Lambda_{k} \cdot \left[\sum_{i=1}^{k-1} \Lambda_{i} \right]  + 2\sum_{k=1}^{N_\epsilon -1} \Lambda_{k} ^2.
\end{align*}
Recall that $N_\epsilon$ be the stopping time that $N_\epsilon = \inf\{k: F(X_k) - F^* \leq \epsilon\}$. Hence for $K = N_\epsilon -1$, we have that $F(X_K) - F^* > \epsilon$ and 
\begin{align*}
    2\Alpha^{succ}_{N_\epsilon -1} T_{N_\epsilon -1}^2 &\leq \frac{3\alpha^{succ}_{0} t_{0}^2v_{0} +2\norm{u_0 }^2 + \sum_{k=1}^{N_\epsilon -1} \Lambda_{k} \cdot \left[\sum_{i=1}^{k-1} \Lambda_{i} \right]  + 2\sum_{k=1}^{N_\epsilon -1} \Lambda_{k}^2}{F(X_{N_\epsilon -1}) - F^*} \nonumber\\
    &\leq \frac{3\alpha^{succ}_{0} t_{0}^2v_{0} +2\norm{u_0 }^2 + \sum_{k=1}^{N_\epsilon} \Lambda_{k} \cdot \left[\sum_{i=1}^{k-1} \Lambda_{i} \right]  + 2\sum_{k=1}^{N_\epsilon} \Lambda_{k} ^2}{\epsilon}.
\end{align*}
Let $\mathcal{N}_G$ be the set of \textit{good} iterations  (i.e. true and successful iterations $k$ with $\alpha_k\geq \stepsizebar $), with $ k \leq N_\epsilon -1$, then $N_G$ is the cardinality of that set. 
Using the inequality $ \alpha^{succ}_{K} t_{K}^2 \geq \left(\sum_{k \in S, k \le K } \sqrt{\alpha_k} /2\right)^2$ for $K =  N_\epsilon-1$ from Lemma \ref{lem:fista_linesearch} we have 
\begin{align*} 
\Alpha^{succ}_{N_\epsilon -1} T_{N_\epsilon -1}^2 
&\geq \left(\sum_{k \in S, k \le N_\epsilon -1} \sqrt{\Alpha_k} /2\right)^2 \nonumber\\
&\geq \left(\sum_{k\in \mathcal{N}_G, k \le N_\epsilon -1 } \sqrt{\Alpha_k} /2\right)^2
\geq \left(N_G \sqrt{\stepsizebar} /2\right)^2 
= \frac{1}{4} (N_G)^2 \stepsizebar.
\end{align*}
Substituting this to the previous inequality we get 
 \begin{align*}
(N_G)^2  &\leq  \frac{2}{\stepsizebar\epsilon} \left [3\alpha^{succ}_{0} t_{0}^2v_{0} +2\norm{u_0 }^2 + \sum_{k=1}^{N_\epsilon} \Lambda_{k} \cdot \left[\sum_{i=1}^{k-1} \Lambda_{i} \right]  + 2\sum_{k=1}^{N_\epsilon} \Lambda_{k} ^2   \right].
\end{align*}
Taking total expectation we have: 
\begin{align*}
\EE [(N_G)^2 ]   &\leq \frac{2}{\stepsizebar\epsilon} \left [3\alpha^{succ}_{0} t_{0}^2v_{0} +2\norm{u_0 }^2 
+ \EE \left[ \sum_{k=1}^{N_\epsilon } \Lambda_{k} \cdot \left(\sum_{i=1}^{k-1} \Lambda_{i} \right)  \right] + 2\EE \left[\sum_{k=1}^{N_\epsilon }  \Lambda_{k} ^2  \right] 
 \right].
\end{align*}
From \eqref{eq:fista_noise} we have: 
\begin{align*}
    \EE \left[\Lambda_{k}^2| \pastone\right] = \EE \left[4\alpha_{k}^2t_k^2 \| (\nabla f(y_{k} )- G_{k} )\|^2 | \pastone\right] \leq \frac{4}{k^{2+ \beta}},
\end{align*}
Hence using a similar argument as Theorem \ref{th:bound_expectation_N_TS1} we get 
\begin{align*}
\EE [(N_G)^2 ]
 &\leq \frac{2}{\stepsizebar\epsilon} \left [3\alpha^{succ}_{0} t_{0}^2v_{0} +2\norm{u_0 }^2  + \frac{4(\beta + 2)^2}{\beta^2} + \frac{8(\beta+2)}{\beta+1} 
 \right].
\end{align*}
By the expectation inequality: 
 \begin{align*}
\EE [N_G] \leq \sqrt{
 \frac{2}{\stepsizebar\epsilon} \left [3\alpha^{succ}_{0} t_{0}^2v_{0} +2\norm{u_0 }^2 + \frac{4(\beta + 2)^2}{\beta^2} + \frac{8(\beta+2)}{\beta+1}  
 \right] }.
\end{align*}
we get the final results of Theorem \ref{th:bound_expectation_N_TS}. 
~\qed \end{proof}

Our final result is a straightforward corollary of Theorem \ref{th:framework} and \ref{th:bound_expectation_N_TS}. 
\begin{theorem}\label{th:final_results}
Under Assumptions \ref{assumpt:convex_smooth}, \ref{assumpt:bounded_below}, \ref{assumpt:key} and Assumption \ref{assumpt:fista_noise} the number of iterations $N_\epsilon$ to reach an $\epsilon$-accurate solution satisfies the following bound: 
\begin{align*}
\EE (N_\epsilon) \leq \frac{2p}{(2p-1)^2}\left ( \left(
 \frac{8 \left [3\alpha^{succ}_{0} t_{0}^2v_{0} +2\norm{u_0 }^2 + \frac{4(\beta + 2)^2}{\beta^2} + \frac{8(\beta+2)}{\beta+1}  
 \right]}{\stepsizebar\epsilon} \right)^{\frac{1}{2}}+ \log_\gamma\left(\frac{\stepsizebar}{\alpha_1}\right)\right) +1,
\end{align*}
where  $\stepsizebar $ is defined in Lemma \ref{lem:good_est_good_model_successful}.
\end{theorem}
\section{Conclusions}\label{sec:summary}

Various stochastic optimization methods with adaptive step-size parameters have been developed and analyzed in recent literature \cite{Blanchet-Cartis-Menickelly-Scheinberg,Cartis-Scheinberg,Jin-Scheinberg-Xie,Paquette-Scheinberg,Scheinberg-Xie}. These works share common features, specifically that stochastic gradient estimates are not assumed to be unbiased, but are simply assumed to be sufficiently accurate with large enough, fixed, probability. This accuracy is dynamically adjusted as the algorithm progresses. The analytical framework of these methods relies on a certain type of martingale analysis. It was unclear, however, if this type of analysis under similar assumptions, could be carried out for accelerated gradient methods and/or for proximal gradient methods in convex composite optimization. We resolve this question positively in this paper. We extend iteration complexity analysis of a backtracking FISTA method to case of stochastic, possibly biased gradient estimates (but accurate function values). We show that the stochastic step-size parameter $\alpha_k$ of our accelerated proximal methods is lower bounded by a random walk with a positive drift, same as is shown in \cite{Blanchet-Cartis-Menickelly-Scheinberg,Cartis-Scheinberg,Jin-Scheinberg-Xie,Paquette-Scheinberg,Scheinberg-Xie}. This, in turn, implies that establishing a high probability lower bound on  $\alpha_k$ using recent results from \cite{Jin-et-al} is straightforward.
Extending this work to the case of stochastic function estimates and obtaining high probability bound of complexity, similar to \cite{Jin-Scheinberg-Xie} is subject of future work. In addition, investigating shuffling momentum methods \cite{Tran-Nguyen-Tran-Dinh,Tran-Scheinberg-Nguyen} in proximal setting is an interesting problem.

\paragraph{Acknowledgement} The authors are grateful to Raghu Bollapragada for useful discussion on the topics of sample complexity of this method and to the anonymous referees for the careful reading of the manuscript and the comments that help us improve it.


\bibliographystyle{plainnat}
\bibliography{refs}





\section{Appendix: Expected complexity vs. expected convergence rate.}

While deterministic counterparts of  our bound on $\mathbb{E}[N_\epsilon]$  and more common bounds on $\mathbb{E}[F(x_k)- F^*]$ are interchangeable,  here we show that the stochastic bounds are not.  
We recall that, for convex setting, expected convergence rate results show the bound $\EE (F(X_k)-F^*)\leq \mathcal{O}(\frac{1}{k})$ while our expected complexity bound implies that $\EE(N_\epsilon)\leq \mathcal{O}(\frac{1}{\epsilon}) $ where $N_\epsilon$ is the hitting time for the event $\{F(X_k)-F^*\leq \epsilon\}$. Below we present small examples of simple nonnegative stochastic processes, for which one of the bounds holds, but not the other. 

\textbf{General stochastic process.}
Let $\{Y_n\}_{n\ge1} $ be some nonnegative sequence of random variables, $\epsilon > 0$ be arbitrarily small.
Then $N_\epsilon$ is the hitting time for the event $\{Y_n\leq \epsilon\}$, respectively. We define 
$\bar{N}_\epsilon$ to be the minimum index $k$ satisfying that $\EE(Y_{k})\leq \epsilon $. If there is no such index then $\bar{N}_\epsilon $ is defined to be $+\infty$.

\textbf{Example 1.} This example demonstrates that $\EE[N_\epsilon]$  finite or bounded by $\mathcal{O}(\frac{1}{\epsilon})$ does not imply that $\bar{N}_\epsilon$ is finite or bounded by $\mathcal{O}(\frac{1}{\epsilon})$. First, we consider the following sequence: 
\begin{align*}
    Y_n = \begin{cases}
        1 &\text{ with prob. } \frac{1}{2},\\
        \epsilon &\text{ with prob. } \frac{1}{2}.
    \end{cases}
\end{align*}
It is easily seen that $\EE[N_\epsilon]$ is finite:
\begin{align*}
    \EE[N_\epsilon] &= \frac{1}{2} \cdot 1 + \frac{1}{2} \cdot (1 + \EE[N_\epsilon]),\quad
    \EE[N_\epsilon] = 2 \leq \infty.
\end{align*}
However the expected value of this sequence never reachs $\epsilon$:
\begin{align*}
    \EE[Y_n] = \frac{1+\epsilon}{2} > \epsilon \text{ when } \epsilon \leq 1, \bar{N}_\epsilon = \infty.
\end{align*}

Similarly, the following sequence satisfies $\EE[N_\epsilon] = \frac{1}{\epsilon}$ and $\EE[Y_n] > \epsilon, \bar{N}_\epsilon = \infty$.
\begin{align*}
    Y_n = \begin{cases}
        1 &\text{ with prob. } 1 -\epsilon,\\
        \epsilon &\text{ with prob. } \epsilon.
    \end{cases}
\end{align*}

\textbf{Example 2.} This example shows that $\bar{N}_\epsilon$  finite or bounded by $\mathcal{O}(\frac{1}{\epsilon})$ does not imply that $\EE[N_\epsilon]$ is finite or bounded by $\mathcal{O}(\frac{1}{\epsilon})$. We consider the following sequence: 
\begin{align*}
    Y_n = \begin{cases}
        \frac{\epsilon}{1-\epsilon} &\text{ with prob. } 1 -\epsilon,\\
        0 &\text{ with prob. } \epsilon.
    \end{cases}
\end{align*}
We have that $\bar{N}_\epsilon$ is finite, i.e. $\EE[Y_n] = \epsilon \leq \epsilon, \bar{N}_\epsilon = 1$; while $\EE[N_\epsilon]$ grows arbitrarily large when $\epsilon$ approaches 0: 
\begin{align*}
    \EE[N_\epsilon] &= \epsilon \cdot 1 + (1 -\epsilon) \cdot (1 + \EE[N_\epsilon]),\quad 
    \EE[N_\epsilon] = \frac{1}{\epsilon}.
\end{align*}

Our final example is more extreme, where $Y_n$ is bounded away from 0 for a small probability: 
\begin{align*}
    Y_n = \begin{cases}
        \frac{1}{2} &\text{ with prob. } \epsilon,\\
        \frac{\epsilon}{1-\epsilon} &\text{ with prob. } \frac{n}{n+1} - \epsilon,\\
        0 &\text{ with prob. } \frac{1}{n+1}.
    \end{cases}
\end{align*}
We have that $\bar{N}_\epsilon=  1$:
\begin{align*}
    \EE[Y_n] &= \frac{\epsilon}{2} + \frac{\epsilon}{1-\epsilon} \left(\frac{n}{n+1}-\epsilon \right),\\
    \EE[Y_n] &\leq \epsilon \text{ when }  \frac{\epsilon}{1-\epsilon} \left(\frac{n}{n+1}-\epsilon \right) \leq \frac{\epsilon}{2}, \text{ satisfied when } n = 1.
\end{align*}
while $\EE[N_\epsilon]$ is infinite: 
\begin{align*}
    \EE[N_\epsilon] &= \frac{1}{2} \cdot 1 + \frac{1}{2} \cdot \left(\frac{1}{3} \cdot 2 + \frac{2}{3} \cdot \left(\frac{1}{4} \cdot 3 + \frac{3}{4} \cdot \left(\frac{1}{5} \cdot 4 + \frac{4}{5} \cdot \left( \frac{1}{6} \cdot 5 + \frac{5}{6} \cdot \left(...  \right) \right)  \right)  \right)  \right)\\
    &= \frac{1}{2} \cdot 1 + \frac{1}{2} \cdot \frac{1}{3} \cdot 2 + \frac{1}{2} \cdot \frac{2}{3} \cdot \frac{1}{4} \cdot 3 + \frac{1}{2} \cdot \left( \frac{2}{3} \cdot \left(\frac{3}{4} \cdot \left(\frac{1}{5} \cdot 4 + \frac{4}{5} \cdot \left( \frac{1}{6} \cdot 5 + \frac{5}{6} \cdot \left(...  \right) \right)  \right)  \right)  \right)\\
    &= \frac{1}{2}  +  \frac{1}{3}  + \frac{1}{4} + \frac{1}{5} +\frac{1}{6} +...\\
    &= \infty.
\end{align*}

\end{document}